\documentclass[11pt]{amsart}
\usepackage{amssymb,mathrsfs,graphicx,enumerate}
\usepackage{amsmath,amsfonts,amssymb,amscd,amsthm,bbm}
\usepackage{graphicx,colortbl}
\usepackage{mathtools}

\title[Discrete CBO algorithm with random batch interactions and heterogeneous noises]{Convergence analysis of the discrete consensus-based optimization algorithm with random batch interactions and heterogeneous noises}

\author{Dongnam Ko}
\address{Department of Mathematics, The Catholic University of Korea, \\
Jibongro 43, Bucheon, Gyeonggido 14662, Republic of Korea\\
dongnamko@catholic.ac.kr}

\author{Seung-Yeal Ha}
\address{Department of Mathematical Sciences and Research Institute of Mathematics, \\ Seoul National University, Seoul 08826 \\
and School of Mathematics, Korea Institute for Advanced Study,\\
Hoegiro 85, Seoul 02455, Republic of Korea\\
syha@snu.ac.kr}

\author{Shi Jin}
\address{School of Mathematical Sciences,  Institute of Natural Sciences, and MOE-LSC,  \\ 
Shanghai Jiao Tong University, Shanghai 200240, China\\
shijin-m@sjtu.edu.cn}

\author{Doheon Kim}
\address{School of Mathematics, Korea Institute for Advanced Study,\\
Hoegiro 85, Seoul 02455, Republic of Korea\\
doheonkim@kias.re.kr}

\newtheorem{theorem}{Theorem}[section]
\newtheorem{lemma}{Lemma}[section]

\newtheorem{proposition}{Proposition}[section]
\newtheorem{remark}{Remark}[section]

\newtheorem{definition}{Definition}[section]

\newcommand{\bbr}{\mathbb R}

\begin{document}

\date{\today}


\thanks{\textbf{Acknowledgment.}
The work of S.-Y. Ha was supported by National Research Foundation of Korea (NRF-2020R1A2C3A01003881), the work of S. Jin was supported by National Natural Science Foundation of China grant 12031013, Shanghai Municipal Science and Technology Major Project 2021SHZDZX0102, and Science and Technology Commission of Shanghai Municipality grant No. 20JC1414100, the work of D. Kim was supported by a KIAS Individual Grant (MG073901) at Korea Institute for Advanced Study, and the work of D. Ko was supported by the Catholic University of Korea, Research Fund, 2021, and by National Research Foundation of Korea (NRF-2021R1G1A1008559).}

\begin{abstract}
We present stochastic consensus and convergence of the discrete consensus-based optimization (CBO) algorithm with {\it random batch interactions} and {\it heterogeneous external noises}. Despite the wide applications and successful performance in many practical simulations, the convergence of the discrete CBO algorithm was not rigorously investigated in such a generality. In this work, we introduce a generalized discrete CBO algorithm with a weighted representative point and random batch interactions, and show that the proposed discrete CBO algorithm exhibits stochastic consensus and convergence toward the common equilibrium state exponentially fast under suitable assumptions on system parameters. For this, we recast the given CBO algorithm with random batch interactions as a discrete consensus model with a random switching network topology, and then we use the mixing property of interactions over sufficiently long time interval to derive stochastic consensus and convergence estimates in mean square and almost sure senses. Our proposed analysis significantly improves earlier works on the convergence analysis of CBO models with full batch interactions and homogeneous external noises. 
\end{abstract}
\maketitle \centerline{\date}

\tableofcontents

\section{Introduction} \label{sec:1}
\setcounter{equation}{0}
Swarm intelligence \cite{B-D-T,Kennedy} provides remarkable meta-heuristic gradient-free optimization methods. Motivated by collective behaviors \cite{A-B-F} of herds, flock, colonies or schools, several population-based meta-heuristic optimization techniques were proposed in literature \cite{Ya, Y-D}. To name a few, particle swarm optimization \cite{K-E}, ant colony optimization \cite{D-B-S}, genetic algorithm \cite{Ho} and grey wolf optimizer \cite{M-M-L} are popular meta-heuristic algorithms which have attracted researchers from various disciplines over the last two decades. Among them, consensus-based optimization (CBO) algorithm \cite{P-T-T-M, C-C-T-T} is a variant methodology of particle swarm optimization, which adopts the idea of consensus mechanism \cite{C-S, Degroot} in a multi-agent system. The mechanism is a  dynamic process for agent (particle) to approach the common consensus via  cooperation \cite{AS-J, C-J-L-Z, C-J-L, F-H-P-S1, F-H-P-S2, H-J-Kc, H-J-Kd, H-K-K-K-Y, J-L-L,  T-P-B-S}. CBO algorithm was further studied as a system of stochastic differential equations \cite{P-T-T-M} and has also been  applied to handle objective functions on constrained search spaces, for example hypersurfaces \cite{F-H-P-S1}, high-dimensional spaces \cite{C-J-L-Z}, the Stiefel manifold \cite{H-K-K-K-Y, K-K-K-H-Y}, etc.  

In this paper, we are interested in stochastic consensus estimates for a discrete CBO algorithm incorporating two  components: ``{\it heterogeneous external noises}" and ``{\it random batch interactions}". In \cite{C-J-L-Z}, a CBO algorithm with these two aspects was studied numerically, while a previous analytical study of CBO in \cite{H-J-Kd} does not cover these aspects. 

We first briefly describe the discrete CBO algorithm in \cite{C-J-L-Z, H-J-Kd}.
Suppose $L:\mathbb R^d \to \mathbb R$ is non-convex objective function  to minimize. As an initial guess, we introduce initial data of particles ${\bold x}_0^i \in \mathbb R^d$ for $i=1,\dots,N$. This is commonly sampled from random variables as in Monte Carlo methods, independent of other system parameters. Then, the following CBO algorithm \cite{C-J-L-Z, H-J-Kd} governs temporal evolution of each sample point ${\bold x}_n^i$ to find the optimal value among the values on the trajectories of particles:
\begin{equation}  \label{A-ooooo}
\begin{cases}
\displaystyle {\bold x}^i_{n+1} = {\bold x}^i_n -\gamma  ({\bold x}^i_n - {\bar {\bold x}}_n^*)  -\sum_{l=1}^{d} (x^{i,l}_n - {\bar x}_n^{*,l})  \eta_{n}^{l} \bold e_l,  \quad   i \in {\mathcal N}, \\
\displaystyle {\bar {\bold x}}_n^* := \frac{\sum_{j=1}^{N} {\bold x}^j_n e^{-\beta L({\bold x}^j_n)}}{\sum_{j=1}^{N} e^{-\beta L({\bold x}^j_n)}}, \quad n \geq 0,  \quad l = 1, \cdots, d,
\end{cases}
\end{equation}
where the random variables $\{ \eta^l_n \}_{n, l}$ are independent and identically distributed(i.i.d.) with
\begin{equation*} 
\mathbb E[\eta_n^l]=0, \quad \mathbb E[|\eta_n^l|^2]=\zeta^2, \quad n \geq 0, \quad l = 1, \cdots, d.
\end{equation*}
 In particular, from the current data of distributed agents, the network interactions between agents share the information of neighbors, and each agent asymptotically approaches to the minimizer of the objective function. Hence, the main role of the CBO algorithm is to make the sample point ${\bold x}_n^i$ converge to the global minimizer ${\bold x}^*$ of $L$. 

In the sequel, we describe a more generalized CBO algorithm which includes \eqref{A-ooooo} as its special case.
Let  ${\bold x}_n^i = (x_n^{i,1},\dots,x_n^{i,d}) \in \mathbb R^d$ be the state of the $i$-th sample point  at time $n$ in the search space, and the set $[i]_n$ denotes an index set of agents interacting with $i$ at time $n$ (neighboring sample points in the interaction network). For notational simplicity, we also set 
\[  {\mathcal X}_n := ( {\bold x}_n^1, \cdots, {\bold x}_n^N ), \quad  {\mathcal N} := \{1, \cdots, N \}. \]
For a nonempty set $S \subset {\mathcal N}$ and an index $j \in {\mathcal N}$, we introduce a weight function $\omega_{S,j}: (\bbr^d)^N \to \bbr_{\geq0}$ which satisfies the following relations:
\begin{equation*}
\omega_{S,j}({\mathcal X})\geq0,\quad \sum_{j\in S} \omega_{S,j}({\mathcal X})=1,\quad \omega_{S,j}({\mathcal X})=0~\mbox{ if }~j\notin S,
\end{equation*}
and the convex combination of $\bold x_n^j$:
\begin{equation} \label{A-2}
 {\bar {\bold x}}_n^{S,*}:=\sum_{j=1}^N \omega_{S,j}({\mathcal X}) {\bold x}_n^j\in\operatorname{conv}\{{\bold x}_n^j:j\in S\},\quad\forall~\varnothing\neq S\subseteq  {\mathcal N}.
\end{equation}
Before we introduce our governing CBO algorithm, we briefly discuss two key components (heterogeneous external noises and random batch interactions) one by one. 

First, a sequence of random matrices $\{ (\eta^{i,l}_n)_{1\leq i\leq N, 1\leq l\leq d} \}_{n\geq0}$ are assumed to be i.i.d. with respect to $n\geq0$, with finite first two moments:
\begin{equation*} \label{A-2-1}
\mathbb E[\eta_n^{i,l}]=0, \quad \mathbb E[|\eta_n^{i,l}|^2]\leq\zeta^2, \quad {\zeta\geq0, }\quad {n \geq0,} \quad l = 1, \cdots, d,\quad i \in {\mathcal N}.
\end{equation*}

Second, we discuss ``{\it random batch method (RBM)}" in  \cite{C-J-L-Z, J-L-L}, to introduce random batch interactions.  For the convergence analysis, we need to clarify how RBM provides a randomly switching network. At the $n$-th time instant, we choose a partition of ${\mathcal N}$ randomly, ${\mathcal B}^n = \{{\mathcal B}_1^n,\ldots,{\mathcal B}_{\lceil \frac{N}{P}\rceil}^n\}$ of $\lceil \frac{N}{P}\rceil$ batches (subsets) with sizes at most $P>1$ as follows:
\[ \{1,\ldots, N\} = {\mathcal B}_1^n \cup {\mathcal B}_2^n \cup \cdots \cup {\mathcal B}_{\lceil \frac{N}{P}\rceil}^n, \quad |{\mathcal B}_i^n| = P, \quad i = 1, \cdots, \Big\lceil \frac{N}{P}\Big\rceil-1, \quad |{\mathcal B}_{\lceil \frac{N}{P}\rceil}^n| \leq P. \]
Let $\mathcal A$ be a set of all such partitions for ${\mathcal N}$. Then, the choices of batches are independent at each time $n$ and follow the uniform law in the finite set of partitions $\mathcal A$. The random variables ${\mathcal B}^n : (\Omega,\mathcal F,\mathbb P) \to \mathcal A$ yield random batches at time $n$ for each realization $\omega \in \Omega$, which are also assumed to be independent of the  initial data $\{{\bold x}_0^i\}_i$ and noise $\{\eta_n^{i,l}\}_{n,i,l}$.
Thus, for each time $n$, the index set of neighbors $[i]_n \in {\mathcal B}^n$ is an element of the partition from ${\mathcal N}$ containing $i$. With the aforementioned key components, we are ready to consider the Cauchy problem to the generalized discrete CBO algorithm under random batch interactions and heterogeneous noises:
\begin{equation}  \label{A-3}
\begin{cases}
\displaystyle {\bold x}^i_{n+1} ={\bold x}^i_n -\gamma  ({\bold x}^i_n - {\bar {\bold x}}_n^{[i]_n,*})  -\sum_{l=1}^{d} (x^{i,l}_n - {\bar x}_n^{[i]_n,*,l})  \eta_{n}^{i,l} \bold   e_l, \quad  n \geq 0, \\
\displaystyle {\mathcal X}_0 = ( {\bold x}_0^1, \cdots, {\bold x}_0^N )\in\mathbb R^{dN}, \quad i \in {\mathcal N},
\end{cases}
\end{equation}
where the initial data ${\mathcal X}_0$ can be deterministic  or random, and the drift rate $\gamma \in (0, 1)$ is a positive constant. In this way, the dynamics \eqref{A-3} includes random batch interactions and heterogeneous noises.

In \cite{H-J-Kd}, the consensus analysis and error estimates for \eqref{A-ooooo} have been done, which is actually the following special case of \eqref{A-3}: for a positive constant $\beta$,
\begin{equation} \label{A-3-1}
{\bar {\bold x}}_n^{[i]_n,*} =  \sum_{j=1}^N   \frac{ e^{-\beta L({\bold x}_n^j)}  }{\sum_{k=1}^{N} e^{-\beta L({\bold x}_n^k)}} {\bold x}_n^j\quad \mbox{and} \quad \eta_{n}^{i,l} = \eta_n^l, \quad  \forall~i \in {\mathcal N}. 
\end{equation}
Here both right-hand sides are independent of $i$, namely the full batch $P=N$ is used, all agent interacts with all the other agents, and the noise is the same for all particles. 
We refer to Section \ref{sec:2.1} for details.

 In this paper, we are interested in the following two questions. \newline
\begin{itemize}
\item 
(Q1:~Emergence of stochastic consensus state)~Is there a common consensus state ${\bold x}_\infty \in \mathbb R^d$ for system \eqref{A-3}, i.e., 
\[ \lim_{n \to \infty} {\bold x}_n^i = {\bold x}_\infty  \quad \mbox{in suitable sense,} \quad \mbox{for all~$i \in {\mathcal N}$}?  \]
\vspace{0.1cm}
\item 
(Q2:~Optimality of the stochastic consensus state)~If so, then is the consensus state ${\bold x}_\infty \in \mathbb R^d$ an optimal point?
\[ L({\bold x}_\infty) = \inf\{L({\bold x}) : {\bold x} \in \mathbb R^d\}. \]
\end{itemize}

\vspace{0.2cm}

Among the above two posed questions, we will mostly focus on the first question (Q1).  In fact,  the continuous analogue for \eqref{A-3}  has been dealt with in literature. In the first paper of CBO \cite{AS-J}, the CBO algorithm without external noise ($\sigma=0$) has been studied, and  the formation of consensus  has been proved under suitable assumption on the network structure. However, the optimality was only shown by numerical simulation even for this deterministic case. In \cite{C-J-L-Z}, the dynamics \eqref{A-3} has been investigated via a mean-field limit ($N \to \infty$). In this case, the particle density distribution, which is the limit of the empirical measure $\frac{1}{N} \sum_{i} \delta_{{\bold x}_n^i}$ in $N \to \infty$, satisfies   a nonlinear Fokker-Planck equation. By using the analytical structure of the Fokker-Planck equation, the consensus estimate  was
established. The optimality of the limit value is also estimated by Laplace's principle, as $\beta \to \infty$ (see Section \ref{sec:2}). Rigorous convergence and error estimates to the continuous and discrete algorithms for particle system \eqref{A-3} with {\it fixed} $N$  were addressed in authors' recent works \cite{H-J-Kc, H-J-Kd} for the special setting \eqref{A-3-1}.

Before we present our main results, we introduce several concepts of stochastic consensus and convergence as follows.
\begin{definition} \label{D1.1}\
Let $\{{\mathcal X}_n \}_{n\geq0}= \{ \bold x_n^i: 1\leq i\leq N \}_{n\geq0}$ be a discrete stochastic process whose dynamics is governed by \eqref{A-3}. Then, several concepts of stochastic consensus and convergence can be defined as follows.
\begin{enumerate}
\item
The random process $\{{\mathcal X}_n \}_{n\geq0}$ exhibits ``consensus in expectation'' or ``almost sure consensus'', respectively if 
\begin{equation} \label{A-4}
\lim_{n \to \infty} \max_{i,j} \mathbb E \|{\bold x}_n^i - {\bold x}_n^j \| = 0 \quad  \mbox{or} \quad 
 \lim_{n \to \infty}  \|{\bold x}_n^i - {\bold x}_n^j \| = 0,~~\forall~~i,j \in {\mathcal N} \quad \mbox{a.s.},
 \end{equation}
 where $\| \cdot \|$ denotes the standard $\ell^2$-norm in ${\mathbb R}^d$. 
 
\vspace{0.2cm}
\item
The random process $\{{\mathcal X}_n \}_{n\geq0}$ exhibits ``convergence in expectation'' or ``almost sure convergence'', respectively if there exists a random vector $\bold x_\infty$ such that 
\begin{equation} \label{A-5}
\lim_{n \to \infty} \max_{i}  \mathbb E \|{\bold x}_n^i - {\bold x}_\infty \|  = 0 \quad  \mbox{or} \quad 
 \lim_{n \to \infty}  \|{\bold x}_n^i - {\bold x}_\infty \| = 0,~~\forall~~i \in {\mathcal N} \quad \mbox{a.s.}
 \end{equation}
\end{enumerate}
\end{definition}
In addition, we introduce some terminologies. For a given set of $N$ vectors $\{ {\bold x}^i \}_{i=1}^{N}$,  we use $X$  to denote the  $N \times d$ matrix whose $i$-th row is $({\bold x}^i)^\top$,
 and the $l$-th column vector $\mathfrak{x}^l$, as
\[ X := \left(
\begin{array}{ccc}
x^{1,1} & \cdots & x^{1,d} \\
\vdots & \vdots & \cdots \\
x^{N,1} & \cdots & x^{N,d}
\end{array}
\right), \quad  {\bold x}^i  := \left(
\begin{array}{c}
x^{i,1} \\
\vdots \\
x^{i,d} 
\end{array}
\right),\quad  {\mathfrak x}^{l}  := \left(
\begin{array}{c}
x^{1,l} \\
\vdots \\
x^{N,l} 
\end{array}
\right), \quad {\mathcal D}(\mathfrak x^l) := \max_{i} x^{i,l} - \min_i x^{i,l}.  \]
The main results of this paper are two-fold. First, we present stochastic consensus estimates for \eqref{A-4} in the sense of Definition \ref{D1.1} (1). For this, we begin with the dynamics of the $l$-th column vector $\mathfrak x^l$: 
\[
{\mathfrak x}_{n+1}^l = (A_n + B_n) {\mathfrak x}_{n}^l, \quad n \geq 0
\]
for some random matrices $A_n$ and $B_n$ corresponding to random source due to, first, only random batch interactions, and second, interplay between the external white noises $\eta_n^{i,l}$ and random batch 
interactions, respectively (see Section \ref{sec:2.1} for details). Then the diameter functional ${\mathcal D}(\mathfrak x^l)$ satisfies 
\begin{equation}\label{A-5-1}
{\mathcal D}(\mathfrak x_{n+1}^l) \leq (1-\alpha(A_n + B_n)){\mathcal D}(\mathfrak x_{n}^l),
\end{equation}
where $\alpha(A)$ is the ergodicity coefficient of matrix $A$ (see Definition \ref{D3.1}). \newline

For the case \eqref{A-3-1} in which interactions are full batch ($P=N$) with a constant weight function $\omega_{\mathcal N,j}$ and the noises are homogenous, one can show that the ergodicity $\alpha(A_n+B_n)$ is strictly positive for small $\zeta$. Then, the recursive relation \eqref{A-5-1} yields desired exponential consensus. Moreover, the assumption \eqref{A-3-1} also implies that the dynamics of $x^{i,l}_n - x^{j,l}_n$ follows a closed form of stochastic difference equation (see \cite{H-J-Kd}). Then, following the idea of geometric Brownian motion from \cite{H-J-Kc}, the decay of $x^{i,l}_n - x^{j,l}_n$ can be obtained.
However, in our setting \eqref{A-3}, the aforementioned blue picture breaks down. In fact, we cannot even guarantee the nonnegativity of $\alpha(A_n + B_n)$. From the definition of ergodic coefficient, the positive ergodicity coefficient implies that any two particles are attracted thanks to network structure, which is not possible when particles are separated by random batches. Hence, the recursive relation \eqref{A-5-1} is not enough to derive an exponential decay of ${\mathcal D}(\mathfrak x_{n}^l)$ as it is.
On the other hand, to quantify {\it enough mixing effect} via the network topology, we use the $m$-transition relation with $m \gg 1$ so that 
\[  0< \alpha \Big( \prod_{k=n}^{n + m -1} (A_k + B_k) \Big) < 1. \]
In this case, the $m$-transition relation satisfies 
\[ {\mathfrak x}_{n+m}^l =  \prod_{k=n}^{n + m -1} (A_k + B_k) {\mathfrak x}_{n}^l, \quad n \geq 0. \]
Again, we use the elementary property of ergodicity coefficient presented in Lemma \ref{L3.3} to find 
\[  {\mathcal D}(\mathfrak x_{n+m}^l) \leq (1-\alpha((A_{n+m-1}+B_{n+m-1})(A_{n+m-2}+B_{n+m-2})\cdots (A_{n}+B_{n}))){\mathcal D}(\mathfrak x_{n}^l). \]
One of the crucial novelties of this paper lies on the estimation technique for a product of matrix summations. When there is no noise $(B_k=0)$, we can use analysis on randomly switching topologies \cite{D-H-J-K2} to derive the positivity of  the ergodic coefficient. In order to handle heterogenous external noises, we adopt the concept of negative ergodicity from general $n$-by-$n$ matrices and treat $B_k$ as a perturbation (see Lemma \ref{L4.1} and Lemma \ref{L4.2} for details). 
This yields an exponential decay of ${\mathcal D}(\mathfrak x^l)$ in suitable stochastic senses for  a sufficiently small $\zeta$, the variance of the noise (see Theorem \ref{T4.1} and Theorem \ref{T4.2}): there exist positive constants $\Lambda_1$, $\Lambda_2$, $C_1$ and a random variable $C_2=C_2(\omega)$ such that 
\begin{align}
\begin{aligned} \label{A-6}
& \mathbb E\mathcal D(\mathfrak x_n^l) \leq C_1 \exp\left(-\Lambda_1 n\right)\quad\text{and}\quad \mathcal D(\mathfrak x_n^l)\leq C_2(\omega) \exp\left(-\Lambda_2 n\right) \quad \mbox{a.s.}
\end{aligned}
\end{align}
This implies
\begin{align*}
\begin{aligned}
& {\mathbb E}\|\bold x_n^i- \bold x_n^j \| \leq C_1 e^{-\Lambda_1 n} \quad\text{and}\quad
\max_{i,j} \|\bold x_n^i- \bold x_n^j \| \leq C_2(\omega) e^{-\Lambda_2 n} \quad \mbox{a.s.} 
\end{aligned}
\end{align*}
Second, we deal with the convergence analysis \eqref{A-5} of stochastic flow $\{ \bold x_n^i \}$. One first sees that the dynamics of $\mathfrak x_n^l$ can be described as 
\begin{equation} \label{A-7}
\mathfrak x_{n}^l-\mathfrak x_{0}^l =-\gamma \sum_{k=0}^{n-1}\left( I_N-  W_k\right)\mathfrak x_k^l  - \sum_{k=0}^{n-1}H_k^l(I_N-W_k)\mathfrak x_k^l,\quad 1\leq l\leq d,~ n\geq0.
\end{equation}
Then, applying the strong law of large numbers as in \cite{H-J-Kd}, the almost sure exponential convergence \eqref{A-6} for $\mathcal D(\mathfrak x_n^l)$ induces the convergence of the first series in the R.H.S. of \eqref{A-7}. On the other hand, the convergence of the second series will be shown using Doob's martingale convergence theorem. Thus, for sufficiently small $\zeta$, there exists a random variable ${\bold x}_\infty$ such that
\[
\lim_{n\to\infty}{\bold x}_n^i= {\bold x}_\infty\quad\mbox{a.s.,}\quad i \in {\mathcal N}.
\]
We refer to Lemma \ref{L5.1} for details. Moreover, the above convergence result can be shown to be exponential in expectation and almost sure senses. A natural estimate yields
\[  |x_{m}^{i,l}-  x_{n}^{i,l}| \leq \sum_{k=n}^{m-1}(\gamma+| \eta_k^{i,l}|)\mathcal D(\mathfrak x_k^{l}), \]
however, the external noises $\eta_n^{i,l}$ are not uniformly bounded in almost sure sense.
Instead, the noises with decay in time, $\eta_n^{i,l}e^{-\varepsilon n}$, is uniformly bounded for any $\varepsilon>0$, as presented in Lemma \ref{L5.2}. This is a key lemma employed in the derivation of exponential convergence: there exists a positive constant $C_3 > 0$ and a random variable $C_4(\omega)$ such that 
\[ {\mathbb E} \| {\bold x}_{\infty}^{i}- {\bold x}_{n}^{i} \|  \leq C_3 e^{-\Lambda_1 n} \quad \mbox{and} \quad \| \bold x_\infty - \bold x_n^i \| \leq C_4(\omega) e^{-\Lambda_2 n} \quad \mbox{a.s.}
\]
(See Theorem \ref{T5.1} for details).  \newline

The rest of this paper is organized as follows. In Section \ref{sec:2}, we briefly present the discrete CBO algorithm \cite{C-J-L-Z} and review consensus and convergence result \cite{H-J-Kd} for the case of homogeneous external noises. In Section \ref{sec:3}, we study stochastic consensus estimates in the absence of external noises so that the randomness in the dynamics is mainly caused by random batch interactions. In Section \ref{sec:4}, we consider the stochastic consensus arising from the interplay of two random sources (heterogeneous external noises and random batch interactions) when the standard deviation of heterogeneous external noise is sufficiently small. In Section \ref{sec:5}, we provide a stochastic convergence analysis for the discrete CBO flow based on the consensus results. In Section \ref{sec:6}, we provide monotonicity of the maximum of the objective function along the discrete CBO flow, and present several numerical simulations to confirm analytical results. Finally, Section \ref{sec:7} is devoted to a brief summary of our main results and some remaining issues to be explored in a future work. 

\vspace{0.5cm}

\noindent {\bf Gallery of notation}:  We use superscript to denote the particle number and components of a vector, for example, $x^{i,k}$ denotes the $k$-th component of the $i$-th's particle, i.e.,
\[ {\bold x}^{i} = (x^{i,1}, \cdots, x^{i,d}) \in {\mathbb R}^d, \quad i \in {\mathcal N}. \]
We also use simplified notation for summation and maximization:
\[ \sum_{i} := \sum_{i=1}^{N}, \quad  \sum_{i,j} := \sum_{i=1}^{N} \sum_{j=1}^{N}, \quad \max_{i} := \max_{1\leq i \leq N}, \quad \max_{i,j} := \max_{1 \leq i, j \leq N}.  \]
For a given state vector ${\mathfrak x} = (x^1, \cdots, x^N) \in \bbr^N$, we introduce its state diameter ${\mathcal D}({\mathfrak x})$:
\begin{equation} \label{A-8}
{\mathcal D}(\mathfrak x) :=\max_{i, j} (x^i-x^j) = \max_{i} x^i - \min_{i} x^i. 
\end{equation}
For $\bold x = (x^1, \cdots, x^d) \in \bbr^d$, we denote its $\ell^\infty$-norm by $\|\bold x\|_\infty := \displaystyle\max_{k}|x^k|$.   \\
Lastly, for integers $n_1\leq n_2$ and square matrices $M_{n_1},M_{n_1+1},\dots, M_{n_2}$ with the same size, we define
\[
\prod_{k=n_1}^{n_2}M_k:=M_{n_2}\cdots M_{n_1+1}M_{n_1}.
\]
Note that the ordering in the product is important, because the matrix product is not commutative in general.
\section{Preliminaries} \label{sec:2}
\setcounter{equation}{0}
In this section, we briefly present the discrete CBO algorithm \eqref{A-3} and then recall convergence and optimization results from  \cite{H-J-Kd}.
\subsection{The CBO algorithms} \label{sec:2.1}
Let  ${\bold x}_t^i = (x_t^{i,1},\dots,x_t^{i,d}) \in \mathbb R^d$ be the state of the $i$-th sample point  at time $t$ in search space, and we assume that state dynamics is governed by the CBO algorithm introduced in \cite{C-J-L-Z}. 
\begin{equation} \label{NB-1}
\begin{cases}
\displaystyle d{\bold x}^i_t = -\gamma ({\bold x}^i_t - {\bar {\bold x}}_t^*)dt  - \sigma  \sum_{l=1}^{d} (x^{i,l}_t - {\bar x}_t^{*,l}) dW_t^{i,l}\bold   e_l, \quad t > 0, \quad i \in {\mathcal N}, \\
\displaystyle {\bar {\bold x}}_t^* ={ (\bar x_t^{*,1}, \cdots,\bar x_t^{*,d}) } := \frac{\sum_{j=1}^{N} {\bold x}^j_t e^{-\beta L({\bold x}^j_t)}}{\sum_{k=1}^{N} e^{-\beta L({\bold x}^k_t)}}.
\end{cases}
\end{equation}
Here $\gamma \in (0,1)$, $\sigma$ and $\beta>0$ are  drift rate, noise intensity and reciprocal of temperature, respectively. $\{ \bold e_l\}_{l=1}^d$ is the standard orthonormal basis in $\mathbb R^d$, and $W_t^{i,l}$ $(l=1,\dots,d,~i=1,\dots,N)$ are i.i.d. standard one-dimensional Brownian motions. \newline

Note the following discrete analogue of Laplace's principle \cite{D-Z}: 
\[
\lim_{\beta\to\infty}\frac{\sum_{j\in {\mathcal N}} {\bold x}^j e^{-\beta L({\bold x}^j)}}{\sum_{j\in {\mathcal N}} e^{-\beta L({\bold x}^j)}}=\mbox{centroid of the set}~\left\{{\bold x}^j ~:~  L({\bold x}^j) = \min_{k \in {\mathcal N}}L({\bold x}^k) \right\}.
\]
This explains why CBO algorithm \eqref{NB-1} is expected to pick up the minimizer of $L$. \newline

Next, we consider time-discretization of \eqref{NB-1}. For this, we set 
\[ h := \Delta t, \quad {\bold x}^i_n := {\bold x}^i_{nh}, \quad n \geq 0, \quad i \in {\mathcal N}. \]
Consider the following three discrete algorithms based on \cite{C-J-L-Z, H-J-Kd}: for $n\geq 0,~~ i = 1, \cdots, N,$
\begin{align*}
\begin{aligned}
& \mbox{Model A}: \quad {\bold x}^i_{n+1} ={\bold x}_n^i -\lambda h ({\bold x}^i_n - {\bar {\bold x}}_n^*)  -  \sum_{l=1}^{d} (x^{i,l}_n - {\bar x}_n^{*,l}) \sigma \sqrt{h}Z_n^l \bold e_l, \\
& \mbox{Model B}: \quad 
\begin{cases}
\displaystyle \hat {\bold x}^i_{n} = {\bar {\bold x}}_n^*+ e^{-\lambda h} ({\bold x}^i_n - {\bar {\bold x}}_n^*),\\
\displaystyle {\bold x}^i_{n+1} =\hat {\bold x}_n^i  -  \sum_{l=1}^{d} (\hat x^{i,l}_n - {\bar x}_n^{*,l}) \sigma \sqrt{h}Z_n^l \bold e_l, 
\end{cases} \\
& \mbox{Model C}: \quad {\bold x}^i_{n+1} ={\bar {\bold x}}_n^*+  \sum_{l=1}^{d} (x^{i,l}_n - {\bar x}_n^{*,l})  \left[\exp\left(-\left(\lambda+\frac{1}{2}\sigma^2\right)h+\sigma\sqrt{h}Z_n^l \right)\right]\bold e_l, 
\end{aligned}
\end{align*}
where the random variables $\{ Z^l_n \}_{n, l}$ are i.i.d  standard normal distributions. Model A is precisely the Euler-Maruyama scheme applied to \eqref{NB-1}, and Models B and C are other discretizations of \eqref{NB-1} introdueced in \cite{C-J-L-Z}. In \cite{H-J-Kd}, it is shown that the three models are special cases (with different choices of $\gamma$ and $\eta_n^l$) of the following generalized scheme:

\begin{equation}  \label{NB-2}
\begin{cases}
\displaystyle {\bold x}^i_{n+1} = {\bold x}^i_n -\gamma  ({\bold x}^i_n - {\bar {\bold x}}_n^*)  -\sum_{l=1}^{d} (x^{i,l}_n - {\bar x}_n^{*,l})  \eta_{n}^{l} \bold e_l,  \quad   i \in {\mathcal N}, \\
\displaystyle {\bar {\bold x}}_n^* := \frac{\sum_{j=1}^{N} {\bold x}^j_n e^{-\beta L({\bold x}^j_n)}}{\sum_{j=1}^{N} e^{-\beta L({\bold x}^j_n)}}, \quad n \geq 0,  \quad l = 1, \cdots, d,
\end{cases}
\end{equation}
where the random variables $\{ \eta^l_n \}_{n, l}$ are i.i.d. with
\begin{equation*} 
\mathbb E[\eta_n^l]=0, \quad \mathbb E[|\eta_n^l|^2]=\zeta^2, \quad n \geq 0, \quad l = 1, \cdots, d.
\end{equation*}
Moreover, discrete algorithm \eqref{NB-2} corresponds to the special case of \eqref{A-3}:
\[ { S} := {\mathcal N}, \qquad \omega_{{\mathcal N}, j} = \frac{e^{-\beta L({\bold x}^j)}}{\sum_{k=1}^{N} e^{-\beta L({\bold x}^k)}}, \quad F_{{\mathcal N}}(X):=\sum_{j=1}^N   \frac{{\bold x}^j e^{-\beta L({\bold x}^j)}  }{\sum_{k=1}^{N} e^{-\beta L({\bold x}^k)}},\quad \eta_n^l\equiv\eta_n^{i,l}.  \]
Therefore, these three models can be explained by the dynamics \eqref{A-3}.

\subsection{Previous results} In this subsection, we briefly summarize the previous results \cite{H-J-Kd} on the convergence and error estimates for the case
\begin{equation*}
\label{2.4}
\eta_n^{i,l}  \equiv \eta_n^l
\end{equation*}
 (thus noises are not heterogeneous). 
In the next theorem, we recall the convergence estimate of the discrete scheme \eqref{NB-2} as follows.
\begin{theorem} \label{T2.1}
\emph{(Convergence estimate \cite{H-J-Kd})} Suppose system parameters $\gamma$ and $\zeta$ satisfy 
\[ (\gamma-1)^2 + \zeta^2 < 1,    \]
and let $\{ {\bold x}_t^i  \}$ be a solution process to \eqref{NB-2}.  Then, one has the following assertions:
\begin{enumerate}
\item
Consensus in expectation and in almost sure sense emerge asymptotically: 
\[ \lim_{n \to \infty} \mathbb E|{\bold x}^{i}_n - {\bold x}^{j}_n|^2 = 0, \quad |x^{i,l}_n - x^{j,l}_n |^2\leq |x^{i,l}_0 - x^{j,l}_0|^2 e^{-n Y_n^l},  \quad \mbox{a.s.}~\omega \in \Omega,
\]
for $ i, j \in {\mathcal N},~~l = 1, \cdots, d,$ and a random variable $Y_n^l$ satisfying
\[ \lim_{n \to \infty} Y_n^l(\omega) = 1 - (\gamma - 1)^2 - \zeta^2 > 0, \quad \mbox{a.s.}~\omega \in \Omega. \]
\item
There exists a common random vector ${\bold x}_\infty=(x_\infty^1,\cdots,x_\infty^d)$ such that
\[  \lim\limits_{n\to\infty}  {\bold x}_n^i= {\bold x}_\infty~\mbox{a.s.}, \quad i \in {\mathcal N}. \]
\end{enumerate}
\end{theorem}

\begin{remark}
In \cite{H-J-Kd}, an error estimate of \eqref{NB-2} toward the global minimum is also obtained as a partial result on optimality. 
If the initial guess ${\bold x}_{in}$ is good enough so that it surrounds the global minimum ${\bold x}_*$ and particles are  close to each other, then the asymptotic consensus state ${\bold x}_\infty$ is close to ${\bold x}_*$.
\end{remark}

\section{Discrete CBO algorithm with random batch interactions}\label{sec:3}
\setcounter{equation}{0}
In this section, we present matrix formulation of \eqref{A-3}, some elementary estimates for ergodicity coefficient and stochastic consensus to system \eqref{A-3} in the absence of heterogeneous external noises. 

\subsection{A matrix reformulation} \label{sec:3.1}
In this subsection, we present a matrix formulation of the following discrete CBO model:
\begin{equation} \label{C-0}
{\bold x}^i_{n+1} = {\bold x}^i_n -\gamma  ({\bold x}^i_n - {\bar {\bold x}}_n^{[i]_n,*})  -\sum_{l=1}^{d} (x^{i,l}_n - {\bar x}_n^{[i]_n,*,l})  \eta_{n}^{i,l} \bold e_l, \quad n \geq 0, \quad  i  \in {\mathcal N}.
\end{equation}
Then, the $l$-th component of \eqref{C-0} can be rewritten as
\begin{equation}  \label{C-1}
x^{i,l}_{n+1} =(1-\gamma-\eta_{n}^{i,l} )x^{i,l}_n  +   (\gamma+\eta_{n}^{i,l} )\sum_{j\in[i]_n} \omega_{[i]_n,j}(X_n) x_n^{j,l}.
\end{equation}
For $n \geq 0$ and $l = 1, \cdots, d$,  set
\begin{equation} \label{C-1-1}
W_n :=(\omega_{[i]_n,j}(X_n)) \in  \bbr^{N \times N}, \quad {\mathfrak x}_n^{l}:=(x_n^{1,l},\dots,x_n^{N,l})^\top, \quad H_n^l :=\operatorname{diag}(\eta_n^{1,l},\dots,\eta_n^{N,l}).
\end{equation}
Note that $ {\mathfrak x}_n^{l}$ denotes the $l$-th column of the matrix $X_n$. Then system \eqref{C-1} can be rewritten in the following matrix form:
\begin{align}\label{C-2}
\begin{aligned}
{\mathfrak x}_{n+1}^l
&=\left[ \underbrace{(1-\gamma)I_N+ \gamma  W_n}_{=: A_n} \underbrace{-H_n^l(I_N-W_n)}_{=:B_n}\right] {\mathfrak x}_n^l = (A_n + B_n) {\mathfrak x}_n^l,\quad n \geq 0, \quad  1\leq l\leq d.
\end{aligned}
\end{align}
The randomness of $A_n$ is due to the random switching of network topology via random batch interactions. In contrast, $B_n$ takes care of the external noises $H_n^l$. 
In the absence of external noise ($\zeta=0$), one has $H_n^l = 0$ for all $n \geq 0$ and system \eqref{C-2} reduces to the stochastic system:
\begin{equation}\label{C-4}
{\mathfrak x}_{n+1}^l = A_n {\mathfrak x}_{n}^l, \quad n \geq 0.
\end{equation}
In the next lemma, we study several properties of $A_n$.
\begin{lemma} \label{L3.1}
Let $A_n$ be a matrix defined in \eqref{C-2}. Then, it satisfies the following two properties:
\begin{enumerate}
\item
$A_n$ is {\it row-stochastic}:
\[    [A_n]_{ij} \geq 0, \quad \sum_{j=1}^{N} [A_n]_{ij} = 1, \quad \forall~i \in {\mathcal N}, \]
where $[A_n]_{ij}$ is the $(i,j)$-th element of the matrix $A_n \in \bbr^{N \times N}$. 

\vspace{0.2cm}

\item
The product $A_{n+ m} \cdots A_n$ is row-stochastic:
\[  [A_{n+ m} \cdots A_n]_{ij} \geq 0, \quad \sum_{j=1}^{d}  [A_{n+ m} \cdots A_n]_{ij} = 1. \]
\end{enumerate}
\end{lemma}
\begin{proof} 
\noindent (1)~By \eqref{C-1-1} and \eqref{C-2}, one has 
\[ 
[A_n]_{ij} = [ (1-\gamma)I_N+ \gamma  W_n ]_{ij} = (1-\gamma) \delta_{ij} + \gamma \omega_{[i]_n,j}(X_n)  \geq 0,
\]
where we used the fact $\gamma \in (0, 1)$. This also yields
\[ \sum_{j} [A_n]_{ij}  = \sum_{j}  (1-\gamma) \delta_{ij} + \gamma \omega_{[i]_n,j}(X_n) = (1-\gamma)  \sum_{j}  \delta_{ij} + \gamma \sum_{j} \omega_{[i]_n,j}(X_n) = 1. \]
\vspace{0.2cm}

\noindent (2)~The second assertion holds since the product of two row-stochastic matrices is also row-stochastic. 

\end{proof}

 Note that the asymptotic dynamics of ${\bold x}_n$ is completely determined by the matrix sequence $\{A_n\}_{n\geq0}$ in \eqref{C-4}. 
\subsection{Ergodicity coefficient} \label{sec:3.2}
In this subsection, we recall the concept of ergodicity coefficient and investigate its basic properties which will be crucially used in later sections. \newline

First, we recall the ergodicity coefficient of a matrix $A \in \bbr^{N \times N}$ as follows.
\begin{definition}  \label{D3.1}
\emph{\cite{A-G}}
For $A=(a_{ij}) \in \bbr^{N \times N}$, we define the ergodicity coefficient of $A$ as 
\begin{equation} \label{C-4-1}
\alpha(A):=\min_{i,j}\sum_{k=1}^N\min\{a_{ik},a_{jk}\}.
\end{equation}
\end{definition}
\begin{remark}
In \cite{C-S}, the ergodicity coefficient \eqref{C-4-1} is introduced to study the emergent dynamics of systems of the form \eqref{C-4}.
\end{remark}
As a direct application of \eqref{C-4-1}, one has the super-additivity and monotonicity for $\alpha$.
\begin{lemma} \label{L3.2}
Let $A, B$ be square matrices in $\bbr^{N \times N}$. Then, the ergodicity coefficient has the following assertions:
\begin{enumerate}
\item
(Super-additivity and homogeneity):
\[ \alpha(A+B)\geq\alpha(A)+\alpha(B) \quad\text{and}\quad \alpha(tA)=t\alpha(A),\quad t\geq0. \]
\item
(Monotonicity):
\[ A\geq B~\mbox{entry-wise} \quad \Longrightarrow \quad \alpha(A)\geq\alpha(B). \]
\end{enumerate}
\end{lemma}
\begin{proof} 
All the estimates can be followed from \eqref{C-4-1} directly. Hence, we omit its details. 
\end{proof}
Next, we recall a key tool that will be crucially used in the consensus estimates for   \eqref{C-2}.
\begin{lemma}\label{L3.3}\cite{A-G}
Let $A \in \bbr^{N \times N}$ be a square matrix with equal row sums:
\[ \sum_{j=1}^{d} a_{ij} = a, \quad \forall~i \in {\mathcal N}. \]
Then, we have
\[
\mathcal D(A\bold z)\leq (a-\alpha(A))\mathcal D(\bold z),\qquad\forall~\bold z\in \bbr^N.
\]
\end{lemma}
\begin{proof} A proof using induction on $N$ can be found in \cite{A-G}, but here we present a different approach. By definition of state diameter \eqref{A-8}, one has
\begin{align*}
\begin{aligned}
{\mathcal D}(A\bold z) &= \max_{i,j} \Big( \sum_{k} a_{ik} z_k -  \sum_{k} a_{jk} z_k ) \\
&= \max_{i,j} \Big(   \sum_{k} ( a_{ik} - \min\{ a_{ik}, a_{jk}\}) z_k - \sum_{k} (a_{jk} - \min\{a_{ik}, a_{jk} \}) z_k  \Big) \\
&\leq \max_{i,j} \Big(  \Big(a - \sum_{k} \min \{ a_{ik}, a_{jk} \}\Big) \max_{k} z_k - \Big(a -  \sum_{k} \min\{a_{ik}, a_{jk} \} \Big) \min_{k} z_k  \Big) \\
&= \max_{i,j} \Big( a - \sum_{k} \min \{ a_{ik}, a_{jk} \} \Big) \Big( \max_{k} z_k - \min_{k} z_k \Big) \\
&= \Big( a - \min_{i,j}  \sum_{k} \min \{ a_{ik}, a_{jk} \} \Big) {\mathcal D}({\bold z}) \\
& = (a-\alpha(A))\mathcal D(\bold z).
\end{aligned}
\end{align*}
\end{proof}
Next, we study elementary lemmas for the emergence of consensus. From Definition \ref{D1.1}, for consensus, it suffices to show
\[  \lim_{n \to \infty} {\mathcal D}(\mathfrak x_{n}^l) = 0, \quad \forall~l. \]
We apply Lemma \ref{L3.3} for \eqref{C-4} and row-stochasticity of $A_n$ to get 
\begin{equation*}\label{C-5}
\begin{aligned}
{\mathcal D}(\mathfrak x_{n+1}^l) \leq (1-\alpha(A_n)){\mathcal D}(\mathfrak x_{n}^l).
\end{aligned}
\end{equation*}
However, note that the ergodicity coefficient $\alpha(A_n)$ may not be strictly positive. For $m \geq 1$, we iterate \eqref{C-4} $m$-times to get
\[
{\mathfrak x}_{n+m}^l = A_{n + m-1} \cdots  A_n {\mathfrak x}_{n}^l, \quad n \geq 0.
\]
Since the product $ A_{n + m-1} \cdots  A_n$ is also row-stochastic (see Lemma \ref{L3.1}), we can apply Lemma \ref{L3.3} to find
\begin{equation*}\label{C-6}
{\mathcal D}(\mathfrak x_{n+m}^l) \leq (1-\alpha(A_{n+m-1}A_{n+m-2}\cdots A_n)){\mathcal D}(\mathfrak x_{n}^l).
\end{equation*}
Then, the ergodicity coefficient $\alpha(A_{n+m-1}A_{n+m-2}\cdots A_n)$ depends on the network structure between $W_{n+m-1}, W_{n+m-2},\ldots, W_{n}$, and the probability of $\alpha$ being strictly positive increases, as $m$ grows.  \newline

In \cite{C-S}, the consensus $\mathcal D(\mathfrak x_n^l)\to0$ $(n\to\infty)$ was guaranteed under the assumption:
\begin{equation}\label{C-7}
\begin{aligned}
\sum_{k=0}^\infty\alpha(A_{t_{k+1}-1}A_{t_{k+1}-2}\dots A_{t_{k} })=\infty\quad\mbox{for some}\quad 0=t_0<t_1<\dots
\end{aligned}
\end{equation}
However, examples of the network topology (which may vary in time, possibly randomly) that satisfies the aforementioned a priori assumption \eqref{C-7} was not well-studied yet. Below, we show that random batch interactions provide a network topology satisfying a priori condition \eqref{C-7}. 

First, we introduce a random variable measuring the degree of network connectivity. For a given pair $(i,j) \in {\mathcal N} \times {\mathcal N}$, we consider the set of all time instant $r$ within the time zone $[n, n+ m)$ such that $i$ and $j$ are in the same batch $([i]_r = [j]_r)$. More precisely, we introduce 
\begin{equation} \label{C-7-1}
{\mathcal T}^{ij}_{[n, n+ m)} := \{ n \leq  r < n + m~:~[i]_r = [j]_r \}, \qquad {\mathcal G}_{[n, n+ m)} := \min_{i,j} \big|{\mathcal T}^{ij}_{[n, n+ m)}\big|. 
\end{equation}
Note that ${\mathcal T}^{ij}_{[n, n+ m)}$ is non-empty if one batch contains both $i$ and $j$ for some instant in $\{n,n+1,\dots,n+ m -1\}$. In the following lemma, we will estimate ergodicity coefficients from random batch interactions.

\begin{lemma}\label{L3.4}
Let $A_n$ be the transition matrix in \eqref{C-2}. Then, for given integers $m\geq0$, $n \geq 0$ and $l=1,\dots,d$, 
\begin{equation} \label{C-8}
\alpha(A_{n+m-1}A_{n+m-2}\cdots A_n) \geq \gamma(1-\gamma)^{m-1} {\mathcal G}_{[n, n+m)}.
\end{equation}
\end{lemma}
\begin{proof} We derive the estimate \eqref{C-8} via two steps:
\begin{equation} \label{C-8-1}
\alpha(A_{n+m-1}A_{n+m-2}\cdots A_n)  \geq    \gamma(1-\gamma)^{m-1}\alpha\Bigg( \sum_{r={n}}^{n+m-1} W_r \Bigg) \geq  \gamma(1-\gamma)^{m-1}  {\mathcal G}_{[n, n+m)}.
\end{equation}

\noindent $\bullet$~Step A (Derivation of the first inequality):~Recall that 
\[ A_n =  (1-\gamma)I_N+ \gamma W_n. \]
Then, it follows from Lemma \ref{L3.2} that 
\begin{align}\label{C-9}
\begin{aligned}
&\alpha(A_{n+m-1}A_{n+m-2}\cdots A_n) \\
& \hspace{0.5cm} = \alpha\Big(\left[(1-\gamma)I_N+ \gamma  W_{n+m-1} \right]\dots\left[(1-\gamma)I_N+ \gamma  W_{n} \right] \Big)\\
&  \hspace{0.5cm} = \alpha\Bigg((1-\gamma)^{m} I_N + \gamma(1-\gamma)^{m-1}\sum_{r={n}}^{n+m-1} W_r + \cdots + \gamma^{m} W_{n+m-1}\cdots W_n\Bigg)  \\
& \hspace{0.5cm}  \geq   \alpha\Bigg( \gamma(1-\gamma)^{m-1}\sum_{r={n}}^{n+m-1} W_r \Bigg)  =  \gamma(1-\gamma)^{m-1}\alpha\Bigg( \sum_{r={n}}^{n+m-1} W_r \Bigg).
\end{aligned}
\end{align}

\vspace{0.2cm}

\noindent $\bullet$~Step B (Derivation of the second inequality):~We use $W_r := (\omega_{[i]_r,j}(X_r))$ and  Definition \ref{D3.1} to see  
\begin{align}
\begin{aligned} \label{C-9-1}
\alpha\Bigg( \sum_{r={n}}^{n+m-1} W_r \Bigg) &=\min_{i,j} \sum_{k=1}^N\min\Bigg\{\sum_{r=n}^{n+m-1} \omega_{[i]_r,k}(X_r),\sum_{r=n}^{n+m-1} \omega_{[j]_r,k}(X_r)\Bigg\}  \\
&\geq\min_{i,j} \sum_{k=1}^N\min\Bigg\{\sum_{\substack{n\leq r < n+m:\\ [i]_r=[j]_r}}  \omega_{[i]_r,k}(X_r),  \sum_{\substack{n\leq r < n+m:\\ [i]_r=[j]_r}} \omega_{[j]_r,k}(X_r)\Bigg\}   \\
&=\min_{i,j} \sum_{k=1}^N \sum_{\substack{n\leq r < n+m:\\ [i]_r=[j]_r}}  \omega_{[i]_r,k}(X_r)
=\min_{i,j} \sum_{\substack{n\leq r < n+m:\\ [i]_r=[j]_r}}  1  \\
&=\min_{i,j}  \left|\{n\leq r < n+m: [i]_r=[j]_r\}\right|.
\end{aligned}
\end{align}
Finally we combine \eqref{C-9} and \eqref{C-9-1} to derive \eqref{C-8-1}. 
\end{proof}
\begin{remark}\label{R3.3}
Below, we provide two comments on the result of Lemma \ref{L3.4}. 
\begin{enumerate}
\item
Lemma \ref{L3.3} says that, if the union of network structure from time $r=n$ to time $r=n+m-1$ forms all-to-all connected network, then the ergodicity coefficient $\alpha(A_{n+m-1}A_{n+m-2}\cdots A_n)$ is positive.
To make the ergodicity coefficient positive, the union network needs not be necessarily all-to-all. However, the estimation on the lower bound of $\alpha$ could be more difficult.
\vspace{0.2cm}
\item
The diameter $D({\mathfrak x}^l_n)$ satisfies 
\begin{equation*} \label{C-9-2}
{\mathcal D}(\mathfrak x_{n+m}^l) \leq \Big(1-\gamma(1-\gamma)^{m-1} {\mathcal G}_{[n, n+m)} \Big) {\mathcal D}(\mathfrak x_{n}^l),\quad n,m\geq0
\end{equation*}
\end{enumerate}
\end{remark}

\subsection{Almost sure consensus} \label{sec:3.3}
In this subsection, we provide a stochastic consensus result when the stochastic noises are turned off, i.e., $\zeta = 0$. For this, we choose a suitable $m$ such that ${\mathcal G}_{[0, m)}$ has a strictly positive expectation. 
\begin{lemma}\label{L3.5}
The following assertions hold.
\begin{enumerate}
\item
There exists $m_0\in\mathbb N$ such that there exist partitions ${\mathcal P}^1,\dots, {\mathcal P}^{m_0}$ of ${\mathcal N}$ in $\mathcal A$, such that for any $i,j \in {\mathcal N}$, one has $i,j\in S \in {\mathcal P}^k$ for some $1\leq k\leq m_{0}$ and $S\in {\mathcal P}^k$.

\vspace{0.2cm}

\item
For  $n \geq 0$, 
\begin{equation}\label{C-10}
 {\mathbb E}\left[{\mathcal G}_{[n,n+m_0)} \right] =  {\mathbb E}\left[{\mathcal G}_{[0, m_0)} \right] =: p_{m_0} > 0.
\end{equation}
\end{enumerate}
\end{lemma}
\begin{proof}
(1)~First, the existence of $m_0$ and partitions are clear since we may consider  all possible partitions ${\mathcal P}^1,\ldots,{\mathcal P}^{m_1}$ of $\{1,\dots,N\}$ in $\mathcal A$. In other words, we may choose $m_0$ as $m_1=|\mathcal A|$.
\newline

\noindent (2)~ Note that the random choices of batches ${\mathcal B}^n$ $(n\geq0)$ are independent and identically distributed. Hence, the expectations of ${\mathcal G}_{[0,m_0)}$ and ${\mathcal G}_{[n,n+m_0)}$ should be identical:
\[ {\mathbb E}\left[{\mathcal G}_{[n,n+m_0)} \right] =  {\mathbb E}\left[{\mathcal G}_{[0,m_0)} \right] =: p_{m_0}. \]
Next, let ${\mathcal P}^1,\dots,{\mathcal P}^{m_0}$ be the partitions chosen in (1). 
Note that the probability of having ${\mathcal B}^1 = {\mathcal P}^1$ depends on $m_1=|\mathcal A|$: for any $i,j$,
\[ {\mathbb P}\left\{ {\mathcal B}^i = {\mathcal P}^j \right\} = 1/m_1. \]
Therefore, we may proceed to give a rough estimate on ${\mathcal G}_{[0,m_0)}$:
\begin{equation*}
\begin{aligned}
{\mathbb E}\left[{\mathcal G}_{[0,m_0)} \right] &\geq {\mathbb P}\left\{{\mathcal G}_{[0,m_0)} \geq 1 \right\}\\
&\geq {\mathbb P}\left\{ \text{for each } 1 \leq j \leq m_0,~ {\mathcal B}^i = {\mathcal P}^j \text{ for some } 0 \leq i \leq m-1 \right\}\\
&\geq \prod_{j=1}^{m_0} {\mathbb P}\left\{ {\mathcal B}^{j-1} = {\mathcal P}^j \right\}
= 1/{(m_1)}^{m_0} >0.
\end{aligned}
\end{equation*}
This gives $p_{m_0}>0.$
\end{proof}
Now we are ready to provide  a stochastic consensus estimate on the dynamics \eqref{C-2} in the absence of noise, $\zeta = 0$.
\begin{theorem} \label{T3.1}
Let ${\{ \bold x}^i_{n}\}$ be a solution process to \eqref{A-3} with $\eta_n^{i,l} = 0$ a.s. for any $n,i,l$, and let  $k \geq 0$. Then, for $n \in [km_0,(k+1)m_0)$, there exists a nonnegative random variable $\Lambda_0 = \Lambda_0(m_0, k)$ such that 
\begin{align} \label{C-11}
\begin{aligned}
& (i)~{\mathcal D}({\mathfrak x}^l_n) \leq  {\mathcal D}({\mathfrak x}^l_0) \exp(-\Lambda_0 k), \quad \mbox{a.s.},  \\
& (ii)~\max_{i,j} \|{\bold x}^i_n - {\bold x}^j_n \|_{\infty}  \leq \max_{i,j} \|{\bold x}^i_0 - {\bold x}^j_0\|_{\infty} \exp(-\Lambda_0 k), \quad \mbox{a.s.,}
\end{aligned}
\end{align}
where the decay exponent $\Lambda_0(m_0,k)$ satisfies 
\begin{equation} \label{C-11-1}
\lim_{k\to\infty} \Lambda_0(m_0,k) = p_{m_0} \gamma(1-\gamma)^{m_0-1} \quad \mbox{a.s.}
\end{equation}
\end{theorem}
\begin{proof}
\noindent (i)~We claim
\begin{equation} \label{C-12}
{\mathcal D}({\mathfrak x^l_{n}}) \leq  {\mathcal D}({\mathfrak x^l_{km_0}}) \leq {\mathcal D}({\mathfrak x^l_{0}}) \exp\Big[ - \gamma(1-\gamma)^{m_0-1} \Big(\frac{1}{k}\sum_{s=1}^k {\mathcal G}_{[(s-1)m_0, sm_0)}\Big) k \Big ].
\end{equation}

\noindent $\bullet$~(First inequality in \eqref{C-12}): Since $A_n$ has only nonnegative elements, the values of $\alpha$ at  products of $A_n$'s are nonnegative. Hence, 
\begin{equation} \label{C-13}
{\mathcal D}({\mathfrak x^l_{n}}) \leq {\mathcal D}({\mathfrak x^l_{km_0}}) \left(1-\alpha(A_{n-1}A_{n-2}\cdots A_{km_0})\right) \leq {\mathcal D}({\mathfrak x^l_{km_0}}).
\end{equation}

\vspace{0.2cm}

\noindent $\bullet$~(Second inequality in \eqref{C-12}): By Lemma \ref{L3.4}, the ergodicity coefficient can be written in terms of ${\mathcal G}_{[(k-1)m_0, km_0)}$:
\begin{align}
\begin{aligned} \label{C-14}
{\mathcal D}({\mathfrak x^l_{km_0}}) &\leq {\mathcal D}({\mathfrak x^l_{(k-1)m_0}}) (1-\alpha(A_{km_0-1}A_{km_0-2}\cdots A_{(k-1)m_0}))\\
&\leq {\mathcal D}({\mathfrak x^l_{(k-1)m_0}}) (1-\gamma(1-\gamma)^{m_0-1}{\mathcal G}_{[(k-1)m_0, km_0)})\\
&\leq {\mathcal D}({\mathfrak x^l_{(k-1)m_0}}) \exp(-\gamma(1-\gamma)^{m_0-1}{\mathcal G}_{[(k-1)m_0, km_0)}),
\end{aligned}
\end{align}
where we used $1+x \leq e^{x}$ in the last inequality. By induction on $k$ in \eqref{C-14}, 
\begin{align} 
\begin{aligned} \label{C-15}
{\mathcal D}({\mathfrak x^l_{km_0}})
&\leq {\mathcal D}({\mathfrak x^l_{0}}) \exp\Big[ -\gamma(1-\gamma)^{m_0-1} \sum_{s=1}^k {\mathcal G}_{[(s-1)m_0, sm_0)}\Big ] \\
&= {\mathcal D}({\mathfrak x^l_{0}}) \exp\Big[ -\gamma(1-\gamma)^{m_0-1} \Big(\frac{1}{k}\sum_{s=1}^k {\mathcal G}_{[(s-1)m_0, sm_0)}\Big)k\Big ].
\end{aligned}
\end{align}
We combine \eqref{C-13} and \eqref{C-15} to get the desired estimate \eqref{C-12}.  Finally, set random variable $\Lambda_0$ as 
\[ \Lambda_0 (m_0,k) := \gamma(1-\gamma)^{m_0-1} \Big(\frac{1}{k}\sum_{s=1}^k {\mathcal G}_{[(s-1)m_0, sm_0)}\Big).  \]

Note that the random variables ${\mathcal G}_{[(s-1)m_0, sm_0)}$ $(s\geq1)$ are independent and identically distributed.
Hence, it follows from the strong law of large numbers that  $\Lambda_0(m_0,k)$ converges to the expectation value as $k \to \infty$, which is estimated in Lemma \ref{L3.4}:
\begin{equation*}
\begin{aligned}
\lim_{k \to \infty} \gamma(1-\gamma)^{m_0-1} \Big(\frac{1}{k}\sum_{s=1}^k {\mathcal G}_{[(s-1)m_0, sm_0)}\Big) = \gamma(1-\gamma)^{m_0-1} {\mathbb E}[{\mathcal G}_{[0,m_0)}] =  \gamma(1-\gamma)^{m_0-1} p_{m_0} \quad \mbox{a.s.}
\end{aligned}
\end{equation*}
This gives the asymptotic property of $\Lambda_0(m_0,k)$ in \eqref{C-11-1}. \newline

\noindent (ii)~By \eqref{A-8} and the definition of ${\mathfrak x}_n^l$, one can see that for $l = 1, \cdots, d$,
\[ {\mathcal D}({\mathfrak x}_n^l) \leq \max_{i,j} \| {\bold x}^i_n - {\bold x}^j_n \|_\infty = \max_{1 \leq k \leq d}  {\mathcal D}({\mathfrak x}_n^k).
\]
This and $\eqref{C-11}_1$ yield $\eqref{C-11}_2$:
\begin{align*}
\begin{aligned}
\max_{i,j}\| {\bold x}^i_n- {\bold x}^j_n\|_{\infty} &= \max_{1\leq l \leq d}{\mathcal D}(\mathfrak x_{n}^l) \leq 
\max_{1 \leq l \leq d}  {\mathcal D}({\mathfrak x}^l_0) \exp(-\Lambda_0 k) 
= \max_{i,j} \| {\bold x}^i_0 - {\bold x}^j_0 \|_{\infty}   \exp(-\Lambda_0 k), \quad \mbox{a.s.} 
\end{aligned}
\end{align*}
\end{proof}
It is also worth mentioning again that the transition matrix $A_n$ has only nonnegative elements. In \cite{C-S}, the nonnegative version of Lemma \ref{L3.3} was applied using the result in  \cite{Markov}. If one takes into account of the randomness $H^l_n$, we need to consider Lemma \ref{L3.3} and generalize Lemma \ref{L3.4} for transition matrices with possibly negative entries. We will analyze it in the next section.

\section{Discrete CBO with random batch interactions and heterogeneous external noises: consensus analysis}\label{sec:4}
\setcounter{equation}{0}
In this section, we study consensus estimates for the stochastic dynamics \eqref{C-2} in the presence of {\it both} random batch intereactions {\it and} heterogeneous external noises. In fact, the materials presented in Section \ref{sec:3} corresponds to the special case of this section. Hence, presentation will be parallel to that of Section \ref{sec:3} and we will focus on the effect of external noises on the stochastic consensus. 

\subsection{Ergodicity coefficient} \label{sec:4.1}
As discussed in Section \ref{sec:3}, we use the ergodicity coefficient to prove the emergence of stochastic consensus. Recall the governing stochastic system:
\begin{align}
\begin{aligned} \label{D-1}
& \mathfrak x^l_{n+1} = (A_n+B_n)\mathfrak x^l_{n}, \\
& A_n := (1-\gamma)I_N+ \gamma W_n \quad\text{and}\quad B_n:= H^l_n(I_N+ W_n).
\end{aligned}
\end{align}
As in Lemma \ref{L3.4}, one needs to compute the ergodicity coefficient of 
\begin{equation}\label{D-2}
(A_{n+m-1}+B_{n+m-1})(A_{n+m-2}+B_{n+m-2})\cdots(A_{n}+B_{n})
\end{equation}
for integers $n \geq 0$ and $m \geq 1$. \newline

In what follows, we study preliminary lemmas for the estimation of ergodicity coefficient $\alpha(\cdot)$ to handle the stochastic part $B_n$ due to external noises as a perturbation. Then, system \eqref{D-1} can be viewed as a perturbation of the nonlinear system \eqref{C-4} which has been treated in the previous section. 
For a given square matrix $A = (a_{ij})$, define a mixed norm:
\[  \|A\|_{1,\infty} :=\max\limits_{1\leq i\leq N}\sum\limits_{j=1}^N|a_{ij}|. \]

\begin{lemma}\label{L4.1}
For a matrix $A=(a_{ij})_{i,j=1}^N \in \bbr^{N \times N}$, the ergodicity coefficient $\alpha(A)$ has a lower bound:
\[
\alpha(A)\geq-2\|A\|_{1,\infty}.
\] 
\end{lemma}
\begin{proof} By direct calculation, one has
\begin{align*}
\alpha(A)&=\min_{i_1,i_2}\sum_{j=1}^N\min\{a_{i_1j},a_{i_2j}\}\geq\min_{i_1,i_2}\sum_{j=1}^N\min\{-|a_{i_1j}|,-|a_{i_2j}|\}\\
&\geq\min_{i_1,i_2}\sum_{j=1}^N(-|a_{i_1j}|-|a_{i_2j}|)=-2\|A\|_{1,\infty}.
\end{align*}
\end{proof}
 The following lemma helps to give a lower bound for the ergodicity coefficient of the product of matrices in \eqref{D-2}.

\begin{lemma}\label{L4.2}
For  $n \in \mathbb N$, let $A_1,\dots,A_n$, $B_1,\dots,B_n$  be matrices in $\bbr^{N \times N}$.  Then, one has
\begin{align}
\begin{aligned} \label{D-3-1}
&\|(A_1+B_1)\cdots(A_n+B_n)-A_1\cdots A_n\|_{1,\infty} \\
&\hspace{2cm} \leq (\|A_1\|_{1,\infty} +\|B_1\|_{1,\infty})\cdots(\|A_n\|_{1,\infty}+\|B_n\|_{1,\infty})-\|A_1\|_{1,\infty} \cdots \|A_n\|_{1,\infty}.
\end{aligned}
\end{align}
\end{lemma}
\begin{proof}
We use the induction on $n$. The initial step $n=1$ is trivial. Suppose that the inequality \eqref{D-3-1} holds for $n-1$. By the subadditivity and submultiplicativity of the matrix norm $\|\cdot\|_{1,\infty}$, we have
\begin{align*}
&\|(A_1+B_1)\cdots(A_n+B_n)-A_1\cdots A_n\|_{1,\infty} \\
&\hspace{1cm} \leq \|(A_1+B_1)\cdots(A_{n-1}+B_{n-1})A_n-A_1\cdots A_n\|_{1,\infty} +\|(A_1+B_1)\cdots(A_{n-1}+B_{n-1})B_n\|_{1,\infty} \\
&\hspace{1cm} \leq \|(A_1+B_1)\cdots(A_{n-1}+B_{n-1}) -A_1\cdots A_{n-1}\|_{1,\infty}\|A_n\|_{1,\infty} \\
&\hspace{1cm} \quad+(\|A_1\|_{1,\infty} +\|B_1\|_{1,\infty})\cdots(\|A_{n-1}\|_{1,\infty}+\|B_{n-1}\|_{1,\infty})\|B_n\|_{1,\infty} \\
&\hspace{1cm} \leq \big( (\|A_1\|_{1,\infty} +\|B_1\|_{1,\infty})\cdots(\|A_{n-1}\|_{1,\infty} +\|B_{n-1}\|_{1,\infty})-\|A_1\|_{1,\infty} \cdots \|A_{n-1}\|_{1,\infty} \big)\|A_n\|_{1,\infty} \\
&\hspace{1cm} \quad+(\|A_1\|_{1, \infty} +\|B_1\|_{1,\infty} )\cdots(\|A_{n-1}\|_{1,\infty} +\|B_{n-1}\|_{1,\infty})\|B_n\|_{1,\infty} \\
&\hspace{1cm} =(\|A_1\|_{1,\infty} +\|B_1\|_{1,\infty})\cdots(\|A_n\|_{1,\infty} +\|B_n\|_{1,\infty})-\|A_1\|_{1,\infty} \cdots \|A_n\|_{1,\infty}.
\end{align*}
This verifies the desired estimate \eqref{D-3-1}. 
\end{proof}

From Lemma \ref{L4.1} and Lemma \ref{L4.2}, we are ready to estimate the ergodicity coefficient of \eqref{D-2}. For this, we also introduce a new random variable
${\mathcal H}_{[n, n+m)}$ measuring the size of error:
\begin{equation}\label{D-4}
{\mathcal H}_{[n, n+m)} := 2\left[\prod_{r=n}^{n+m-1}(1+2\|H^l_r\|_{1,\infty})-1\right],
\end{equation}
where $H_n^l :=\operatorname{diag}(\eta_n^{1,l},\dots,\eta_n^{N,l})$.  
\begin{lemma}\label{L4.3}
For given integers $m \geq 1$, $n \geq 1$ and $l=1,\dots,d$,  one has
\begin{equation} \label{D-4-1}
\alpha \left( \prod_{r=n}^{n+m-1}(A_r+B_r) \right) \geq \gamma(1-\gamma)^{m-1}{\mathcal G}_{[n,n+m)} - {\mathcal H}_{[n,n+m)},
\end{equation}
where ${\mathcal G}_{[n, n+m)}$ is defined in \eqref{C-7-1}.
\end{lemma}
\begin{proof}
First, we use super-additivity of $\alpha$ in Lemma \ref{L3.2}:
\begin{align}
\begin{aligned} \label{D-5}
\alpha\left( \prod_{r=n}^{n+m-1}(A_r+B_r) \right) &\geq \alpha\left( \prod_{r=n}^{n+m-1}A_r \right) + \alpha\left( \prod_{r=n}^{n+m-1}(A_r+B_r)- \prod_{r=n}^{n+m-1}A_r \right) \\
&=: {\mathcal I}_{11} + {\mathcal I}_{12}.
\end{aligned}
\end{align}
\vspace{0.2cm}

\noindent $\bullet$~(Estimate of ${\mathcal I}_{11}$):  This case has already been treated in Lemma \ref{L3.4}:
\begin{equation}\label{D-6}
\alpha\left( \prod_{r=n}^{n+m-1}A_r \right) \geq \gamma(1-\gamma)^{m-1} {\mathcal G}_{[n, n+m)}.
\end{equation}

\noindent $\bullet$~(Estimate of ${\mathcal I}_{12}$): The term ${\mathcal I}_{12}$ can be regarded as an error term due to external stochastic noise. We may use Lemma \ref{L4.1} to get a lower bound of this term:
\begin{equation}\label{D-7}
\begin{aligned}
\alpha\left( \prod_{r=n}^{n+m-1}(A_r+B_r)- \prod_{r=n}^{n+m-1}A_r \right) \geq -2\left\| \prod_{r=n}^{n+m-1}(A_r+B_r)- \prod_{r=n}^{n+m-1}A_r \right\|_{1,\infty}.
\end{aligned}
\end{equation}
By Lemma \ref{L4.2}, one gets
\begin{equation}\label{D-8}
\begin{aligned}
\left\| \prod_{r=n}^{n+m-1}(A_r+B_r)- \prod_{r=n}^{n+m-1}A_r \right\|_{1,\infty} \leq \prod_{r=n}^{n+m-1}(\|A_r\|_{1,\infty} +\|B_r\|_{1,\infty}) - \prod_{r=n}^{n+m-1}\|A_r\|_{1,\infty}.
\end{aligned}
\end{equation}
By the defining relations \eqref{C-2}, $A_r$ and $B_r$ can be estimated as follows.
\begin{equation}\label{D-9}
\begin{aligned}
&\|A_r\|_{1,\infty} = \|(1-\gamma)I_N+ \gamma  W_r\|_{1,\infty} \leq 1,\\
&\|B_r\|_{1,\infty} = \|H^l_r(I_N-W_r)\|_{1,\infty} \leq \|H^l_r\|_{1,\infty} (\|I_N\|_{1,\infty}+\|W_r\|_{1,\infty}) = 2\|H^l_r\|_{1,\infty}.\\
\end{aligned}
\end{equation}
Now, we combine \eqref{D-7}, \eqref{D-8} and \eqref{D-9} to estimate ${\mathcal I}_{12}$:
\begin{equation}\label{D-10}
\begin{aligned}
{\mathcal I}_{12} \geq -2\left[\prod_{r=n}^{n+m-1}(1+2\|H^l_r\|_{1,\infty})-1\right] = - {\mathcal H}_{[n,n+m)}.
\end{aligned}
\end{equation}
Finally, combining \eqref{D-5}, \eqref{D-6} and \eqref{D-10} yield the desired estimate:
\begin{equation*}
\alpha\left( \prod_{r=n}^{n+m-1}(A_r+B_r) \right) \geq \gamma(1-\gamma)^{m-1}{\mathcal G}_{[n, n+m)} - {\mathcal H}_{[n,n+m)}.
\end{equation*}
\end{proof}
\begin{remark}
For the zero noise case $\zeta = 0$, the random variable ${\mathcal H}_{[n, n+m)}$ becomes zero and the estimate \eqref{D-4-1} reduces to \eqref{C-8}. 
\end{remark}
\subsection{  Stochastic consensus } \label{sec:4.2} 
Recall the discrete system for ${\mathfrak x}^l_{n}$:
\[  \mathfrak x^l_{n+1} = (A_n+B_n)\mathfrak x^l_{n}. \]
By iterating the above relation $m$ times, one gets 
\begin{equation} \label{D-10-1}
\mathfrak x^l_{n+m} = \Big(  \prod_{r=n}^{n+m-1}(A_r+B_r) \Big) \mathfrak x^l_{n}. 
\end{equation}
\begin{lemma}\label{L4.4}
Let ${\{ \bold x}^i_{n}\}$ be a solution process to \eqref{A-3} and let $m\geq1$ and $k\geq0$ be given integers. Then, for $n\in[km,(k+1)m)$ and $l=1,\dots,d$,
\begin{equation*}
\begin{aligned}
\mathcal D(\mathfrak x_n^l)
&\leq{\mathcal D}(\mathfrak x_{0}^l)\left( 1+{\mathcal H}_{[km,n)} \right)\prod_{s=1}^{k}\Big( 1 - \gamma(1-\gamma)^{m-1}{\mathcal G}_{[(s-1)m, sm)} + {\mathcal H}_{[(s-1)m, sm)} \Big).
\end{aligned}
\end{equation*}
\end{lemma}
\begin{proof} We will follow the same procedure employed in Theorem \ref{T3.1}, i.e., we first bound $\mathcal D(\mathfrak x_n^l)$ using ${\mathcal D}(\mathfrak x_{km}^l)$, and then we bound ${\mathcal D}(\mathfrak x_{km}^l)$ using ${\mathcal D}(\mathfrak x_{0}^l)$. \newline

\noindent $\bullet$~Step A:  For $n\in[km,(k+1)m)$, we estimate ${\mathcal D}(\bold x_{n}^l)$ by using ${\mathcal D}(\bold x_{km}^l)$:
\begin{equation}\label{D-10-1-0}
\begin{aligned}
{\mathcal D}(\mathfrak x_{n}^l) &\leq \left(1-\alpha \left( \prod_{r=km}^{n-1}(A_r+B_r) \right)\right){\mathcal D}(\mathfrak x_{km}^l) \\
 &\leq \left( 1-\gamma(1-\gamma)^{n-km-1}{\mathcal G}_{[km, n)}+{\mathcal H}_{[km, n)} \right){\mathcal D}(\mathfrak x_{km}^l)\\
&\leq \left( 1+{\mathcal H}_{[km, n)} \right){\mathcal D}(\mathfrak x_{km}^l),
\end{aligned} 
\end{equation}
where we used the fact that ${\mathcal G}_{[km, n)} \geq 0$ in the last inequality. 

\vspace{0.2cm}

\noindent $\bullet$~Step B: We claim
\begin{equation} \label{D-10-1-1}
{\mathcal D}(\mathfrak x_{km}^l) \leq \prod_{s=1}^{k}\Big(1-  \gamma(1-\gamma)^{m-1}{\mathcal G}_{[(k-1)m, km)} - {\mathcal H}_{[(k-1)m, km)} \Big){\mathcal D}(\mathfrak x_{0}^l).  
\end{equation}
{\it Proof of claim:} By setting $n = (k-1) m$ in \eqref{D-10-1},  one has
\[  \mathfrak x^l_{km} = \Big( \prod_{r=(k-1)m}^{km-1} (A_r+B_r) \Big) {\mathfrak x}^l_{(k-1)m}. \]
Then, we use Lemma \ref{L3.1} (with $a=1$) to \eqref{D-10-1} to obtain
\begin{equation} \label{D-10-2}
{\mathcal D}(\mathfrak x_{km}^l) \leq \left(1-\alpha \left( \prod_{r=(k-1)m}^{km-1}(A_r+B_r) \right) \right){\mathcal D}(\mathfrak x_{(k-1)m}^l).
\end{equation}
By induction on $k$, the relation \eqref{D-10-2} yields
\begin{equation}\label{D-11}
\begin{aligned}
{\mathcal D}(\mathfrak x_{km}^l)&\leq \prod_{s=1}^{k}\left(1-\alpha \left( \prod_{r=(s-1)m}^{sm-1}(A_r+B_r) \right)\right){\mathcal D}(\mathfrak x_{0}^l). 
\end{aligned}
\end{equation}
On the other hand, it follows from Lemma \ref{L4.3} that 
\begin{equation}\label{D-12}
\alpha \left( \prod_{r=(k-1)m}^{km-1}(A_r+B_r) \right) \geq \gamma(1-\gamma)^{m-1}{\mathcal G}_{[(k-1)m, km)} - {\mathcal H}_{[(k-1)m, km)}.
\end{equation}
Finally, one combines \eqref{D-11} and \eqref{D-12} to derive \eqref{D-10-1-1}. \newline

\vspace{0.2cm}

\noindent $\bullet$~Final step: By combining \eqref{D-10-1-0} and \eqref{D-10-1-1}, one has
\begin{align*}
\begin{aligned} \label{D-13-1}
{\mathcal D}(\mathfrak x_{n}^l)
&\leq \left( 1+{\mathcal H}_{[km, n)} \right){\mathcal D}(\mathfrak x_{km}^l)\\
&\leq {\mathcal D}(\mathfrak x_{0}^l)\left( 1+{\mathcal H}_{[km,n)} \right)\prod_{s=1}^{k}\Big( 1 - \gamma(1-\gamma)^{m-1}{\mathcal G}_{[(s-1)m, sm)} + {\mathcal H}_{[(s-1)m, sm)} \Big).
\end{aligned}
\end{align*}
\end{proof}
Now, we are ready to provide an exponential decay of ${\mathbb E} {\mathcal D}(\mathfrak x_{n}^l)$.
\begin{theorem}\label{T4.1}
\emph{($L^1(\Omega)$-consensus)}
Let   $\{ {\bold x}^i_n \}$ be a solution process to \eqref{A-3}, and let $m_0$ be a positive constant defined in Lemma \ref{L3.5}. Then there exists $\zeta_1>0$ independent of the dimension $d$, such that if $0\leq\zeta<\zeta_1$, there exists some positive constants $C_1 = C_1(N, m_0, \zeta)$ and $\Lambda_1=\Lambda_1(\gamma,N,m_0,\zeta)$ such that 
\[ \mathbb E\mathcal D(\mathfrak x_n^l)\leq  C_1 \mathbb E\mathcal D(\mathfrak x_{0}^l)\exp\left(-\Lambda_1 n\right),\quad n \geq 0. \]
In particular,
\[ \max_{i,j} \|\bold x_n^i- \bold x_n^j \|_\infty \leq  C_1 \max_{i,j} \|\bold x_0^i- \bold x_0^j \|_\infty\exp\left(-\Lambda_1 n\right),\quad n \geq 0. \]
\end{theorem}
\begin{proof} 
\noindent  Assume $n\in[km_0,(k+1)m_0)$ $(k\geq0)$. It follows from Lemma \ref{L4.4} that 
\begin{equation} \label{D-13-2}
\mathcal D(\mathfrak x_n^l) \leq {\mathcal D}(\mathfrak x_{0}^l)\left( 1+{\mathcal H}_{[km_0, n)} \right)\prod_{s=1}^{k}\Big( 1 - \gamma(1-\gamma)^{m_0-1}{\mathcal G}_{[(s-1)m_0, sm_0)} + {\mathcal H}_{[(s-1)m_0, sm_0)} \Big).
\end{equation}
Since the following random variables
\[ {\mathcal D}(\mathfrak x_{0}^l), \quad {\mathcal H}_{[km_0, n)}, \quad {\mathcal G}_{[(s-1)m_0, sm_0)}, \quad  {\mathcal H}_{[(s-1)m_0, sm_0)}, \quad s = 1, \cdots, k \]
are independent, we can take expectations on both sides of the estimate in \eqref{D-13-2} and then use the inequality $1+x\leq e^x$ to see that 
\begin{align}\label{D-14}
\begin{aligned}
{\mathbb E}{\mathcal D}(\mathfrak x_{n}^l) &\leq\mathbb E {\mathcal D}(\mathfrak x_{0}^l)\left( 1+\mathbb E{\mathcal H}_{[km_0, n)} \right)\prod_{s=1}^{k}\Big( 1 - \gamma(1-\gamma)^{m_0-1}\mathbb E{\mathcal G}_{[(s-1)m_0, sm_0)} +\mathbb E {\mathcal H}_{[(s-1)m_0, sm_0)} \Big)\\
&\leq {\mathbb E}{\mathcal D}(\mathfrak x_{0}^l)\left(1+ {\mathbb E} {\mathcal H}_{[km_0,n)} \right)    \exp\bigg[ - \gamma  (1-\gamma)^{m_0-1 }\sum_{s=1}^k{\mathbb E}{\mathcal G}_{[(s-1)m_0, sm_0)}
+ \sum_{s=1}^{k} {\mathbb E}{\mathcal H}_{[(s-1)m_0, sm_0)} \bigg].
\end{aligned}
\end{align}
In the sequel, we estimate the terms ${\mathbb E}{\mathcal G}_{[(s-1)m_0, sm_0)} $ and ${\mathbb E}{\mathcal H}_{[(s-1)m_0, sm_0)}$ one by one. \newline

\noindent $\bullet$~Case A (Estimate of ${\mathbb E}{\mathcal H}_{[(s-1)m_0, sm_0)}$): Recall  the defining relation ${\mathcal H}_{[n,n+m_0)}$ in \eqref{D-4}:
\[
{\mathcal H}_{[(s-1)m_0, sm_0)} := 2\left[\prod_{(s-1)m_0}^{sm_0-1}(1+2\|H^l_r\|_{1,\infty})-1\right].
\]
By taking  expectation of ${\mathcal H}_{[(s-1)m_0, sm_0)}$,
\begin{equation}\label{D-15}
\begin{aligned}
{\mathbb E} {\mathcal H}_{[(s-1)m_0, sm_0)}
&= 2\left[ \prod_{r=(s-1)m_0}^{sm_0-1} \left( 1+2{\mathbb E} \| H^l_r \|_{1,\infty} \right) -1 \right]
= 2\left[ \left( 1+2{\mathbb E} \| H^l_0 \|_{1,\infty} \right)^{m_0} -1 \right]\\
&\leq 2\left[ \left( 1+2\sqrt{N} \zeta \right)^{m_0} -1 \right],
\end{aligned}
\end{equation}
where we used the fact ${\mathbb E} |x| \leq \sqrt{{\mathbb E} |x|^2}$ to get
\[  {\mathbb E} \| H_n^l \|_{1,\infty} ={\mathbb E}  \max_{k} |\eta_n^{k,l} | \leq {\mathbb E} \sqrt{  \sum_{k} |\eta^{k,l}_n|^2 } \leq \sqrt{ {\mathbb E} \sum_{k} |\eta_{n}^{k,l}|^2         } \leq \sqrt{N} \zeta. \]

\vspace{0.5cm}

\noindent $\bullet$~Case B (Estimate of ${\mathbb E}{\mathcal G}_{[(s-1)m_0, sm_0)}$):~By \eqref{C-10} of Lemma \ref{L3.5}, one has 
\[ {\mathbb E}{\mathcal G}_{[(s-1)m_0, sm_0)} = p_{m_0}. \]

 Combining \eqref{D-14} and \eqref{D-15}  yields 
\begin{equation}\label{D-16}
\begin{aligned}
{\mathbb E}{\mathcal D}(\mathfrak x_{n}^l) 
&\leq {\mathbb E}\mathcal D(\mathfrak x_{0}^l) (1+2[ (1+2\sqrt{N} \zeta)^{n-km_0}-1])\\
&\quad\times\exp\left\{ -k\gamma(1-\gamma)^{m_0-1}p_{m_0} +2k\left[\left(1 +2\sqrt{N}\zeta\right)^{m_0}-1\right]\right \}.
\end{aligned}
\end{equation}
Note that our goal is to estimate ${\mathbb E}{\mathcal D}(\mathfrak x_{n}^l)$ in terms of $n$. Then, $k$ is actually a function of $n$. To be more specific,
\[ k=k(n)=\Big\lfloor\frac{n}{m_0}\Big\rfloor. \]
If $n \geq m_0$, then
\begin{equation}\label{D-17}
\frac{k}{n} \geq \frac{1}{n}\left(\frac{n}{m_0}-1+\frac{1}{m_0}\right)=\frac{1}{m_0}-\frac{1}{n}\left(1-\frac{1}{m_0}\right)\geq\frac{1}{m_0}-\frac{1}{m_0}\left(1-\frac{1}{m_0}\right)=\frac{1}{m_0^2}.
\end{equation}
If $1\leq n\leq m_0$, then
\begin{equation}\label{D-117}
\frac{m_0}{n}\geq1
\end{equation}
Therefore, by combining \eqref{D-16}, \eqref{D-17}, \eqref{D-117} one has 
\[
\mathbb E\mathcal D(\mathfrak x_n^l)\leq \exp\left(-\widehat \Lambda_n\cdot n\right)\mathbb E\mathcal D(\mathfrak x_{0}^l)(1+2[ (1+2\sqrt{N}\zeta)^{m_0-1}-1])e^{m_0},\quad n \geq 1,
\]
where the sequence $\{\widehat \Lambda_n\}_{n\geq 1}$ satisfies
\begin{equation*}
\begin{aligned}
\widehat \Lambda_n &:= \frac{k}{n}\Big(\gamma(1-\gamma)^{m_0-1}p_{m_0} -2\left[\left(1 +2\sqrt{N}\zeta\right)^{m_0}-1\right] \Big) +\frac{m_0}{n}\\
&\geq \min\left\{\frac{1}{m_0^2}\Big(\gamma(1-\gamma)^{m_0-1}p_{m_0} -2\left[\left(1 +2\sqrt{N}\zeta\right)^{m_0}-1\right] \Big) ,1 \right\}\\
&= \frac{1}{m_0^2}\Big(\gamma(1-\gamma)^{m_0-1}p_{m_0} -2\left[\left(1 +2\sqrt{N}\zeta\right)^{m_0}-1\right] \Big) =: \Lambda_1(\gamma,N,m,\zeta)>0,
\end{aligned}
\end{equation*}
provided that $\zeta$ is sufficiently small. Hence
\[ \mathbb E\mathcal D(\mathfrak x_n^l)\leq \exp\left(-\Lambda_1\cdot n\right)\mathbb E\mathcal D(\mathfrak x_{0}^l)(1+2[ (1+2\sqrt{N} \zeta)^{m_0-1}-1])e^{m_0},\quad n \geq 0. \]
Now by setting
\[ C_1 :=(1+2[ (1+2\sqrt{N} \zeta)^{m_0-1}-1])e^{m_0} \]
one gets the desired result.
\end{proof}
In the next theorem, we derive almost sure consensus of \eqref{A-3}.
\begin{theorem}\label{T4.2}
\emph{(Almost sure consensus)}
Let ${\{ \bold x}^i_{n}\}$ be a solution process to \eqref{A-3}, and let $m_0$ be a positive constant defined in Lemma \ref{L3.5}. Then there exists $\zeta_2 >0$ independent of the dimension $d$, such that if $0\leq\zeta<\zeta_2$ and $l \in \{1, \cdots, d\}$, then there exist a positive constant $\Lambda_2=\Lambda_2(\gamma,N,m,\zeta)$ and random variables $C_2^l=C_2^l(\omega)$ and $C_2 = \displaystyle\max_{1\leq l\leq d}C_2^l>0$ such that
\begin{align*}
\begin{aligned}
& (i)~\mathcal D(\mathfrak x_n^l)\leq C^l_2(\omega) \exp\left(-\Lambda_2 n\right), \quad n \geq 0, \quad \quad \mbox{a.s.} \\
& (ii)~\max_{i,j} \|\bold x_n^i- \bold x_n^j \|_\infty \leq C_2 \exp\left(-\Lambda_2 n\right),\quad n\geq 0,\quad \mbox{a.s.}
\end{aligned}
\end{align*}
\end{theorem}
\begin{proof}
(i)~
We apply the inequality $1+x\leq e^x$ to Lemma \ref{L4.4} to get
\begin{align}\label{D-18}
\begin{aligned}
&\mathcal D(\mathfrak x_n^l) \leq {\mathcal D}(\mathfrak x_{0}^l) \left( 1+{\mathcal H}_{[km_0, n)} \right)
\exp\bigg[ - \gamma  (1-\gamma)^{m_0-1 }\sum_{s=1}^k{\mathcal G}_{[(s-1)m_0, sm_0)}
+ \sum_{s=1}^{k} {\mathcal H}_{[(s-1)m_0, sm_0)} \bigg]\\
&\leq{\mathcal D}(\mathfrak x_{0}^l)\exp\bigg[ - \gamma  (1-\gamma)^{m_0-1 }\sum_{s=1}^k{\mathcal G}_{[(s-1)m_0, sm_0)}
+ \sum_{s=1}^{k+1} {\mathcal H}_{[(s-1)m_0, sm_0)} \bigg]\\
&={\mathcal D}(\mathfrak x_{0}^l)\exp\Bigg[ - \gamma (1-\gamma)^{m_0-1 }\Bigg(\frac{1}{k}\sum_{s=1}^k{\mathcal G}_{[(s-1)m_0, sm_0)} \Bigg)\cdot\frac{k}{n}\cdot n \\
&\hspace{4cm}+ \Bigg(\frac{1}{k+1}\sum_{s=1}^{k+1} {\mathcal H}_{[(s-1)m_0, sm_0)}\Bigg)\cdot\frac{k+1}{n}\cdot n  \Bigg].
\end{aligned}
\end{align}
As in Theorem \ref{T4.1}, one may consider 
\[ k=k(n)=\Big\lfloor\frac{n}{m_0}\Big\rfloor. \]
Then, its limiting behavior is :
\begin{equation}\label{D-19}
\lim_{n\to\infty}k=\infty\quad\mbox{and}\quad\lim_{n\to\infty}\frac{k}{n}=\frac{1}{m_0}.
\end{equation}
Hence, one can apply the strong law of large numbers to get
\begin{equation}\label{D-20}
\lim_{n\to\infty}\frac{1}{k}\sum_{s=1}^k{\mathcal G}_{[(s-1)m_0, sm_0)}=\mathbb E\left[{\mathcal G}_{[0, m_0)}\right] = p_{m_0}\quad\mbox{a.s.}
\end{equation}
and
\begin{equation}\label{D-21}
\begin{aligned}
&\lim_{n\to\infty}\frac{1}{k+1}\sum_{s=1}^{k+1} {\mathcal H}_{[(s-1)m_0, sm_0)}  \\
& \hspace{1cm} =\lim_{n\to\infty}\frac{1}{k+1}\sum_{s=1}^{k+1} 2\left[\left(1 +2\|H_{s m_0-1}^l\|_{1,\infty}\right)\cdots\left( 1+2\|H_{(s-1)m_0}^l\|_{1,\infty} \right)-1\right]  \\
& \hspace{1cm} = 2[(1+2\mathbb E\|H_0^l\|_{1,\infty})^{m_0}-1]   \quad \mbox{a.s.}
\end{aligned}
\end{equation}
Finally, combining \eqref{D-18}, \eqref{D-19}, \eqref{D-20} and \eqref{D-21} gives
\[
\mathcal D(\mathfrak x_n^l)\leq \exp\left(-\widetilde \Lambda_n\cdot n\right)\mathcal D(\mathfrak x_{0}^l),\quad n\geq1,
\]
where the random process $(\widetilde \Lambda_n)_{n\geq1}$ is defined by
\begin{equation*}
\begin{aligned}
\widetilde \Lambda_n := - \gamma (1-\gamma)^{m_0-1 }\Bigg(\frac{1}{k}\sum_{s=1}^k{\mathcal G}_{[(s-1)m_0, sm_0)} \Bigg)\cdot\frac{k}{n}
+ \Bigg(\frac{1}{k+1}\sum_{s=1}^{k+1} {\mathcal H}_{[(s-1)m_0, sm_0)}\Bigg)\cdot\frac{k+1}{n},
\end{aligned}
\end{equation*}
and satisfies
\begin{equation*}
\begin{aligned}
\lim_{n\to\infty}\widetilde \Lambda_n
&=\frac{\gamma (1-\gamma)^{m_0-1 }}{m_0} \mathbb E\left[{\mathcal G}_{[(s-1)m_0, sm_0)}\right]-\frac{2[(1+2\mathbb E\|H_0^l\|_{1,\infty})^{m_0}-1]}{m_0}\\
&\geq \frac{\gamma (1-\gamma)^{m_0-1 }p_{m_0}}{m_0}  -\frac{2[(1+2\sqrt{N}\zeta)^{m_0}-1]}{m_0}.
\end{aligned}
\end{equation*}
Thus, if $\zeta$ is sufficiently small,
\[ \lim_{n\to\infty}\widetilde \Lambda_n > 0. \]
This implies that, if one chooses a positive constant $\Lambda_2= \Lambda_2(\gamma,N,m_0,\zeta)$ such that 
\begin{equation}\label{D-22-0}
0 < \Lambda_2 < \frac{\gamma (1-\gamma)^{m_0-1 }p_{m_0}}{m_0}  -\frac{2[(1+2\sqrt{N}\zeta)^{m_0}-1]}{m_0},
\end{equation}
there exists a stopping time $T>0$ such that
\[ \mathcal D(\mathfrak x_n^l)\leq \exp\left(-\Lambda_2 n\right)\mathcal D(\mathfrak x_{0}^l), \quad n \geq T. \]
In other words, there exists an almost surely positive and bounded random variable 
\[ C_2^l(\omega) := \max_{0\leq n < T}\{ \exp\left(\Lambda_2 n\right)\mathcal D(\mathfrak x_n^l) \}, \] 
which satisfies
\begin{equation} \label{D-22}
\mathcal D(\mathfrak x_n^l)\leq C^l_2(\omega) \exp\left(-\Lambda_2 n\right), \quad n \geq 0.
\end{equation}
\newline

\noindent (ii)~In the course of proof of Theorem \ref{T3.1}, we had
\begin{equation} \label{D-23}
\max_{i,j} \| {\bold x}^i_n - {\bold x}^j_n \|_\infty = \max_{1 \leq l \leq d}  {\mathcal D}({\mathfrak x}_n^l) \quad \mbox{a.s.}
\end{equation}
Combining \eqref{D-22} and \eqref{D-23} gives
\[
\max_{i,j}\| {\bold x}^i_n- {\bold x}^j_n\|_\infty = \max_{1\leq l \leq d} {\mathcal D}(\mathfrak x_{n}^l) \leq 
\Big( \max_{1 \leq l \leq d} C^l_2(\omega)   \Big) \exp(-\Lambda_2 n) =: C_2(\omega) \exp(-\Lambda_2 n), \quad \mbox{a.s.} 
\]
\end{proof}
\section{Convergence of discrete CBO flow} \label{sec:5}
\setcounter{equation}{0}
In this section, we present the convergence analysis of the CBO algorithm \eqref{A-3} using the stochastic consensus estimates in Theorems \ref{T4.1} and \ref{T4.2}. In those theorems, it turns out that consensus of \eqref{A-3}, i.e., decay of the relative state
$\|{\bold x}^i - {\bold x}^j \|$,  is shown to be exponentially fast in expectation and in almost sure sense. 
However, this does not guarantee that a solution converges toward an equilibrium of the discrete model \eqref{A-3}, 
since it may approach a limit cycle or exhibit chaotic trajectory even if the relative states decay to zero. The convergence of the states $\mathfrak x_n^i$ to a common point is an important issue for the purpose of global optimization of the CBO algorithm. 

\begin{lemma}\label{L5.1}
\emph{(almost sure convergence)}
Let ${\{ \bold x}^i_{n}\}$ be a solution process to \eqref{A-3}. Then, there exists $\zeta_2>0$ independent of the dimension $d$, such that if $0\leq\zeta<\zeta_2$,  there exists a random variable ${\bold x}_\infty$ such that
the solution $\{{\bold x}^i_n\}$ of \eqref{A-3} converges to ${\bold x}_\infty$ almost surely:
\[
\lim_{n\to\infty}{\bold x}_n^i= {\bold x}_\infty\quad\mbox{a.s.,}\quad i \in {\mathcal N}.
\]
\end{lemma}
\begin{proof}
We will use almost sure consensus of the states ${\bold x}_n^i$ ($i=1,\dots,N$) in Theorem \ref{T4.2} and similar arguments of Theorem 3.1 in \cite{H-J-Kd}.  Similar to \eqref{C-2}, we may rewrite \eqref{A-3} as
\begin{align}\label{E-1}
\begin{aligned}
\mathfrak x_{n+1}^l- \mathfrak x_{n}^l
&=-\gamma\left( I_N- W_n\right) \mathfrak x_n^l-H_n^l(I_N-W_n)\mathfrak x_n^l,\quad 1\leq l\leq d,~ n\geq0,
\end{aligned}
\end{align}
where the convergence of ${\bold x}^i_n$ for all $i=1,\ldots,N$ is equivalent to that of $\mathfrak x_{n}^l$ for all $l=1,\ldots,d$. Summing up \eqref{E-1} over $n$ gives
\begin{equation}\label{E-2}
\mathfrak x_{n}^l-\mathfrak x_{0}^l =-\gamma \underbrace{\sum_{k=0}^{n-1}\left( I_N-  W_k\right)\mathfrak x_k^l}_{=:{\mathcal J}_1}  -\underbrace{\sum_{k=0}^{n-1}H_k^l(I_N-W_k)\mathfrak x_k^l}_{=: {\mathcal J}_2},\quad 1\leq l\leq d,~ n\geq0.
\end{equation}
Note that the convergence of $\mathfrak x_n^l$ follows if the two vector series in the R.H.S. of \eqref{E-2} are both convergent almost surely as $n\to\infty$.  \newline

\noindent $\bullet$~Case A (Almost sure convergence of ${\mathcal J}_1$): For the $i$-th standard unit vector $\bold e_i$ in $\mathbb R^N$, one has
\begin{align}\label{E-3}
\begin{aligned}
|\bold e_i^\top\left(I_N - W_k\right)\mathfrak x_k^l| &= \left|x_k^{i,l}-\sum_{j=1}^N \omega_{[i]_k,j}(X_k^1,\dots,X_k^N)x_k^{j,l}\right|\\
&\leq \sum_{j=1}^N \omega_{[i]_k,j}(X_k^1,\dots,X_k^N) \left|x_k^{i,l}- x_k^{j,l}\right | \leq \max_{1\leq j\leq N}\left|x_k^{i,l}-x_k^{j,l}\right| = \mathcal D(\mathfrak x_k^{l}).
\end{aligned}
\end{align}
where we used the fact that $ \sum_j \omega_{[i]_k,j} = 1$.  It follows from Theorem \ref{T4.2} that the diameter process $\mathcal D(\mathfrak x_k^{l})$ decays to zero exponentially fast almost surely. Hence, for each $i=1,\dots,N$, the relation \eqref{E-3} implies
\[
\sum_{k=0}^{\infty}|\bold e_i^\top\left(I_N -  W_k\right)\mathfrak x_k^l| \leq\sum_{k=0}^{\infty}\mathcal D(\mathfrak x_k^{l})<\infty,\quad\mbox{a.s.}
\]
This shows that each component of the first series in \eqref{E-2} is convergent almost surely. \newline

\noindent $\bullet$~Case B (Almost sure convergence of ${\mathcal J}_2$):~Note that this series is martingale since the random vectors $\{H_k^l(I_N-W_k)\mathfrak x_k^l \}_{k\geq0}$ are independent and have zero mean: for any $\sigma$-field $\mathcal F$ independent of $H^l_k$, we have
\begin{equation*} \label{E-3-1}
{\mathbb E}[~ H_k^l(I_N-W_k)\mathfrak x_k^l ~|~ \mathcal F ~] =  {\mathbb E}[~ H_k^l(I_N-W_k)\mathfrak x_k^l ]=  \Big( {\mathbb E} H_k^l \Big) \cdot \Big(  {\mathbb E} (I_N-W_k) \Big) \cdot 
\Big(  {\mathbb E} \mathfrak x_k^l \Big) = 0,
\end{equation*} 
where we used the fact that ${\mathbb E} H_k^l  = 0$. 
By Doob's martingale convergence theorem \cite{Du}, the martingale converges almost surely if the series is uniformly bounded in $L^2(\Omega)$. Again, it follows from \eqref{E-3} and Theorem \ref{T4.2} that
\begin{align*}
\begin{aligned}
&\sup_{n\geq0}\mathbb E\left|\sum_{k=0}^{n}\bold e_i^\top H_k^l(I_N-W_k)\mathfrak x_k^l\right|^2  =\sup_{n\geq0}\mathbb E\left|\sum_{k=0}^{n}\eta_k^{i,l}\bold e_i^\top (I_N-W_k)\mathfrak x_k^l\right|^2 \\
& \hspace{1cm} \leq \sup_{n\geq0}\mathbb E\sum_{k=0}^{n}\left|\eta_k^{i,l}\bold e_i^\top (I_N-W_k)\mathfrak x_k^l\right|^2 =\sup_{n\geq0}\sum_{k=0}^{n}{\mathbb E}|\eta_k^{i,l}|^2\mathbb E\left|\bold e_i^\top (I_N-W_k)\mathfrak x_k^l\right|^2
\\
&  \hspace{1cm} \leq\zeta^2\sum_{k=0}^{\infty}\mathbb E\left|\bold e_i^\top (I_N-W_k)\mathfrak x_k^l\right|^2 \leq\zeta^2\sum_{k=0}^{\infty}\mathbb E|\mathcal D(\mathfrak x_k^{l})|^2
<\infty.
\end{aligned}
\end{align*}
This yields that the second series in \eqref{E-3} converges almost surely.

 In \eqref{E-3}, we combine results from Case A and Case B to conclude that $\mathfrak x^l_n$ converges to a random variable almost surely.
\end{proof}

\begin{lemma}\label{L5.2}
Let $\{Y_k\}_{k\geq0}$ be an i.i.d. sequence of random variables with $\mathbb E|Y_0|<\infty$. Then,  for any given $\delta>0$,  there exists a random variable $\bar C(\omega)$ such that
\[
|Y_k|< \bar C(\omega)e^{\delta k},\quad\forall~ k\geq0, \quad \mbox{a.s.}
\]
\end{lemma}
\begin{proof}
By Markov's inequality, for any constant $n\geq1$ one has
\[
\mathbb P(|Y_k|\geq ne^{\delta k})\leq \frac{\mathbb E|Y_k|}{ne^{\delta k}},\quad\forall~k\geq0.
\]
In particular, 
\begin{align*}
\begin{aligned} 
\mathbb P(|Y_k|< ne^{\delta k},~\forall~ k\geq0)&= 1- \mathbb P(|Y_k|\geq ne^{\delta k},~\exists k\geq0)
\geq 1- \sum_{k=0}^N \mathbb P(|Y_k|\geq ne^{\delta k})
\geq 1- \frac{\mathbb E|Y_k|}{n(1-e^{-\delta})}.
\end{aligned}
\end{align*}
Therefore, for any $\delta>0$, 
\begin{equation*}\label{E-3-2}
\mathbb P(\exists~n\geq1~\mbox{s.t.}~\mbox{for}~\forall~ k\geq0,~|Y_k|< ne^{\delta k})=\lim_{n \to \infty} \mathbb P(|Y_k|< ne^{\delta k},~\forall~ k\geq0) = 1.
\end{equation*}
This implies existence of a random variable $\bar C(\omega)$ satisfying
\[
|Y_k|< \bar C(\omega)e^{\delta k},\quad\forall~ k\geq0, \quad \mbox{a.s.}
\]
In details, we may define the random variable $\bar C(\omega)$ as
\[ \bar C(\omega) := \sum_{n=1}^\infty \mathbf{1}_{\bar C_n}, \quad \bar C_n :=\Omega \setminus \left\{\omega \in \Omega ~:~ |Y_k(\omega)| < ne^{\delta k},~\forall~ k\geq0 \right\}, \quad n \geq 1. \]
For example, if $\omega$ satisfies $|Y_k(\omega)| < m  e^{\delta k},~\forall~ k\geq0$ for some $m$ but not for $(m-1)$, then $\omega \in \bar C_1,\bar C_2,\dots,\bar C_{m-1}$ and $\omega \notin \bar C_m,\bar C_{m+1},\dots$ so that $\bar C(\omega) = m$.
\end{proof}

Now, we are ready to present the convergence of $\bold x^i_n$ in expectation and almost sure sense. 
\begin{theorem}\label{T5.1}
Let ${\{ \bold x}^i_{n}\}$ be a solution process to \eqref{A-3}, and let $m_0$ be a positive constant defined in Lemma \ref{L3.5}. Then, there exists $\zeta_3>0$ independent of the dimension $d$, such that if $0\leq\zeta<\zeta_3$,  then there exists some random variable $\bold x_\infty$ such that ${\bold x}^i_n$ exponentially converges to ${\bold x}_\infty$ in the following senses:
\begin{enumerate}
\item (Convergence in expectation): for positive constants $C_1$ and $\Lambda_1$ appearing in Theorem \ref{T4.1}, one has
\[ {\mathbb E} \| {\bold x}_{n}^{i}- {\bold x}_{\infty}^{i} \|_{1}  \leq  \frac{d(\gamma + \sqrt{N} \zeta)  C_1 \mathbb E\mathcal D(\mathfrak x_{0}^l)}{1-e^{-\Lambda_1}} e^{-\Lambda_1 n}, \quad n \geq 0,~~ i \in {\mathcal N}. \]
\item (Almost sure convergence): for a positive constant $\Lambda_2$ appearing in Theorem \ref{T4.2}, there exists a positive random variable $C_3(\omega)>0$ such that 
\[
\| \bold x_n - \bold x_\infty^i \|_1 \leq C_3(\omega) e^{-\Lambda_2 n} \quad \mbox{a.s.} \quad n \geq 0,~~ i \in {\mathcal N}.
\]
\end{enumerate}
\end{theorem}
\begin{proof} Let $m_0$ be a positive integer appearing in Theorem \ref{T4.1}. Then, for $n \geq 0$, there exists $\zeta_1>0$ independent of the dimension $d$, such that if $0\leq\zeta<\zeta_1$, there exists some positive constants $C_1 = C_1(N, m_0, \zeta)$ and $\Lambda_1=\Lambda_1(\gamma,N,m_0,\zeta)$ such that 
\begin{equation} \label{E-3-3}
\mathbb E\mathcal D(\mathfrak x_n^l)\leq  C_1 \mathbb E\mathcal D(\mathfrak x_{0}^l)\exp\left(-\Lambda_1 n\right),\quad n \geq 0.
\end{equation}
Let $M \geq n$. Then, summing up \eqref{E-1} over $n$ gives
\begin{align*}
\begin{aligned}
\mathfrak x_{M}^l-\mathfrak x_{n}^l
&=-\gamma\sum_{k=n}^{M-1}\left( I_N-  W_k\right)\mathfrak x_k^l-\sum_{k=n}^{M-1}H_k^l(I_N-W_k)\mathfrak x_k^l,\quad 1\leq l\leq d,~ n\geq 0.
\end{aligned}
\end{align*}
As in the proof of Lemma \ref{L5.1}, for the $i$-th standard unit vector $\bold e_i$, one has
\begin{equation}\label{E-4}
\begin{aligned}
 |x_{M}^{i,l}-  x_{n}^{i,l}|&=|\bold e_i^\top(\mathfrak x_{M}^l-\mathfrak x_{n}^l)|
\leq\gamma\sum_{k=n}^{M-1}|\bold e_i^\top\left( I_N- W_k\right)\mathfrak x_k^l|+\sum_{k=n}^{M-1}|\eta_k^{i,l}||\bold e_i^\top  (I_N-W_k)\mathfrak x_k^l|\\
&\leq \sum_{k=n}^{M-1}(\gamma+| \eta_k^{i,l}|)\mathcal D(\mathfrak x_k^{l}),
\end{aligned}
\end{equation}
where the last inequality follows from $\eqref{E-3}$. \newline

\noindent (i) Taking expectations to \eqref{E-4} and use \eqref{E-3-3} gives
\begin{align*}
\begin{aligned} 
{\mathbb E}|x_{M}^{i,l}-  x_{n}^{i,l}|
&\leq \sum_{k=n}^{M-1} (\gamma + {\mathbb E}|\eta_k^{i,l}|)\mathbb E[\mathcal D(\mathfrak x_k^{l})]
\leq (\gamma + \sqrt{N} \zeta) C_1 \mathbb E\mathcal D(\mathfrak x_{0}^l) \sum_{k=n}^{M-1}e^{-\Lambda_1 k} \\
&\leq \frac{(\gamma + \sqrt{N} \zeta) C_1 \mathbb E\mathcal D(\mathfrak x_{0}^l) e^{-\Lambda_1 n}}{1-e^{-\Lambda_1}}.
\end{aligned}
\end{align*}
This yields
\begin{equation} \label{E-4-1}
{\mathbb E} \| {\bold x}_{M}^{i}- {\bold x}_{n}^{i} \|_\infty =\max_{1\leq l \leq d} {\mathbb E}|x_{M}^{i,l}-  x_{n}^{i,l}| \leq  \frac{(\gamma + \sqrt{N} \zeta)  C_1}{1-e^{-\Lambda_1}} \max_{1\leq l \leq d}\mathbb E\mathcal D(\mathfrak x_{0}^l) e^{-\Lambda_1 n}.
\end{equation}
In \eqref{E-4-1}, letting  $M \to\infty$  and using Lemma \ref{L5.1} give
\[ {\mathbb E} \| {\bold x}_{\infty}^{i}- {\bold x}_{n}^{i} \|_\infty  \leq  \frac{(\gamma + \sqrt{N} \zeta)  C_1}{1-e^{-\Lambda_1}} \max_{1\leq l \leq d}\mathbb E\mathcal D(\mathfrak x_{0}^l) e^{-\Lambda_1 n}, \quad n \geq 0. \]


\vspace{0.5cm}

\noindent (ii)~
In the proof of Theorem \ref{T4.2}, the constant $\Lambda_2$ was chosen in \eqref{D-22-0} so that
\[ 0 < \Lambda_2 < \frac{\gamma (1-\gamma)^{m_0-1 }p_{m_0}}{m_0}  -\frac{2[(1+2\sqrt{N}\zeta)^{m_0}-1]}{m_0}. \]
Here one may fix $\Lambda_2$ and consider small $\varepsilon>0$ satisfying
\[ 0 < \Lambda_2 < \Lambda_2 + \varepsilon < \frac{\gamma (1-\gamma)^{m_0-1 }p_{m_0}}{m_0}  -\frac{2[(1+2\sqrt{N}\zeta)^{m_0}-1]}{m_0}. \]
Then, from Theorem \ref{T4.2} and Lemma \ref{L5.2}, the estimates of $\mathcal D(\mathfrak x_n^l)$ and $|\eta_k^{i,l}|$ can be described by
\begin{align*}
\begin{aligned}
& \mathcal D(\mathfrak x_n^l)\leq C^l_2(\omega) \exp\left(-(\Lambda_2+\varepsilon) n\right), \quad n \geq 0, \quad \quad \mbox{a.s.}, \\
&  |\eta_k^{i,l}|< \bar C(\omega)e^{\varepsilon k},\quad\forall~ k\geq0, \quad \mbox{a.s.}
\end{aligned}
\end{align*}
We apply these estimates to \eqref{E-4} and obtain
\begin{align*}
\begin{aligned}
 |x_{M}^{i,l}-  x_{n}^{i,l}|
&\leq \sum_{k=n}^{M-1}(\gamma + \bar C(\omega) e^{\varepsilon k}) C_2^l(\omega) e^{-(\Lambda_2+\varepsilon)k}
\leq {\widetilde C}(\omega) \sum_{k=n}^{M-1}e^{-\Lambda_2 k}\leq \frac{{\widetilde C}(\omega) e^{-\Lambda_2 n}}{1-e^{-\Lambda_2}},\quad\mbox{a.s.} 
\end{aligned}
\end{align*}
for some positive and bounded random variable ${\widetilde C}(\omega)$, independent of $l$. Now, sending $M \to \infty$ and using Lemma \ref{L5.1} give
\[  |x_{\infty}^{l}-  x_{n}^{i,l}| \leq \frac{{\widetilde C}(\omega) e^{-\Lambda_2 n}}{1-e^{-\Lambda_2}},\quad\mbox{a.s.} \]
This yields the desired estimate:
\[ \| \bold x_\infty - \bold x_n^i \|_\infty \leq  \frac{  {\widetilde C}(\omega) }{1-e^{-\Lambda_2}}e^{-\Lambda_2 n},\quad\mbox{a.s.}  \quad n \geq 0. \]

\end{proof}


\section{Application to optimization problem} \label{sec:6}
\setcounter{equation}{0}
In this section, we discuss how the consensus estimates studied in previous section can be applied to the optimization problem. Let $L$ be an objective function that we look for a minimizer of \newline
\[ \min_{\bold x \in \bbr^d} L(\bold x). \]
For the weighted representative point \eqref{A-2} with a nonempty set $S \subset {\mathcal N}$, we set
\begin{equation} \label{F-1}
 {\bar {\bold x}}_n^{S,*}:= \bold x_n^i\quad\mbox{with}~i=\min \{ \operatorname{argmin}_{j\in S}L({\bold x_n^j}) \}.
\end{equation}
In this setting, we specify one agent of the set $\operatorname{argmin}_{X^j:j\in S}L(X^j)$ \cite{C-J-L-Z}.
\begin{proposition}\label{P6.1}
Let ${\{ \bold x}^i_{n}\}$ be a solution process to \eqref{A-3} equipped with \eqref{F-1}. Then, for all $n\geq0$, the best guess among all agents at time $n$ has monotonicity of its objective value:
\[
\min_{1 \leq j \leq N}L(\bold x_{n+1}^j)\leq  \min_{1\leq j\leq N}L(\bold x_n^j).
\]
\end{proposition}
\begin{proof}
For a fixed $n$, let $i$ be the index satisfying 
\[ i=\min \{ \operatorname{argmin}_{1\leq j\leq N}L(\bold x_n^j) \}. \]
Then, it is easy to see 
\[
i=\min \{ \operatorname{argmin}_{j\in[i]_n}L(\bold x_n^j) \},\quad \mbox{i.e.,} \quad  {\bar {\bold x}}_n^{[i]_n,*} = \bold x_n^i.
\]
 For each $1\leq l\leq d$, we have
\begin{align*}
x^{i,l}_{n+1} &=(1-\gamma-\eta_{n}^{i,l} )x^{i,l}_n  +   (\gamma+\eta_{n}^{i,l} )\sum_{j\in[i]_n} \omega_{[i]_n,j}(\bold x_n^1,\dots, \bold x_n^N) x_n^{j,l}  \\
&=(1-\gamma-\eta_{n}^{i,l} )x^{i,l}_n  +   (\gamma+\eta_{n}^{i,l} ) x_n^{i,l} =x_n^{i,l} ,
\end{align*}
i.e., 
\[ \bold x_{n+1}^i= \bold x_n^i. \]
Hence, one has
\[
\min_{1\leq j\leq N}L(\bold x_{n+1}^j)\leq L(\bold x_{n+1}^i)=L(\bold x_{n}^i)=\min_{1\leq j\leq N}L(\bold x_n^j).
\]
\end{proof}
\begin{remark} Proposition \ref{P6.1} suggests that from the relation
\[
\min_{1\leq j\leq N}L(\bold x_n^j)=\min_{0\leq t\leq n,~1\leq j\leq N}L(\bold x_t^j),
\]
the monotonicity actually gives
\[
{\bar {\bold x}}_n^{\mathcal N,*}\in\operatorname{argmin}_{\bold x_n^j,~1\leq j\leq N}L(\cdot)=\operatorname{argmin}_{\bold x_n^j,~0\leq t\leq n,~1\leq j\leq N}L(\cdot).
\]
In other words, ${\bar {\bold x}}_n^{\mathcal N,*}$ is a best prediction for the minimizer of $L$, among the trajectories of all agents on time interval $[0,~n]$. In the particle swarm optimization (PSO, in short) literature \cite{K-E}, such point is called `global best', and this observation tells us that the algorithm \eqref{A-3} equipped with \eqref{F-1} may be considered as a simplified model of the PSO algorithm.
\end{remark}

Next, we provide numerical simulations for the CBO alglorithms with \eqref{F-1}. From Proposition \ref{P6.1}, this algorithm has monotonicity related to the objective function, which is useful to locate the minimum of $L$.
The CBO algorithms \eqref{A-3}--\eqref{F-1} will be tested with the Rastrigin function as in \cite{P-T-T-M}, which reads
\[ 
L(x) = \frac{1}{d}\sum_{i=1}^d \left[ (x_i-B_i)^2 - 10\cos(2\pi(x_i-B_i)) + 10 \right] + C,
\]
where $\mathbf{B}=(B_1,\dots,B_d)$ is the global minimizer of $L$. We set $C=0$ and $B_i=1$ for all $i$. The Rastrigin function has a lot of local minima near $B$, whose values are very close to the global minimum. The parameters are chosen similar to \cite{C-J-L-Z,P-T-T-M}:
\[ N = 100, \quad \gamma = 0.01, \quad \zeta = 0.5, \quad P=10,50, \quad d=2,3,4,5,6,7,8,9,10, \]
where the initial positions of the particles are distributed uniformly on the box $[-3,3]^d$.

Note that the objective function $L$ only affects the weighted average $\bar{\bold x}_n^{S,*}$ in the CBO dynamics \eqref{A-3}. 
Theorems \ref{T4.1} and \ref{T4.2} show that $\gamma$ and $\zeta$ mainly determine (a lower bound of) speed of convergence. One useful property of the CBO algorithm is that it always produces a candidate solution for the optimization problem (see Theorem \ref{T5.1}). However, the candidate may not even be a local minimum.
Therefore, following \cite{P-T-T-M}, we consider one simulation successful if the candidate is close to the global minimum $B$ in the sense that
\[ \|\bold x_n^i - \mathbf{B}\|_\infty < 0.25, \quad\mbox{with}\quad i=\min \left\{ \operatorname{argmin}_{1\leq j\leq N}L({\bold x_n^j}) \right\}.\]
The stopping criterion is made with the change of positions,
\[ \sum_{i=1}^N|\bold x_{n+1}^i - \bold x_{n}^i|^2 < \varepsilon, \quad \varepsilon = 10^{-3}. \]

\begin{figure}[h]
    \centering
    \includegraphics[width=0.45\textwidth]{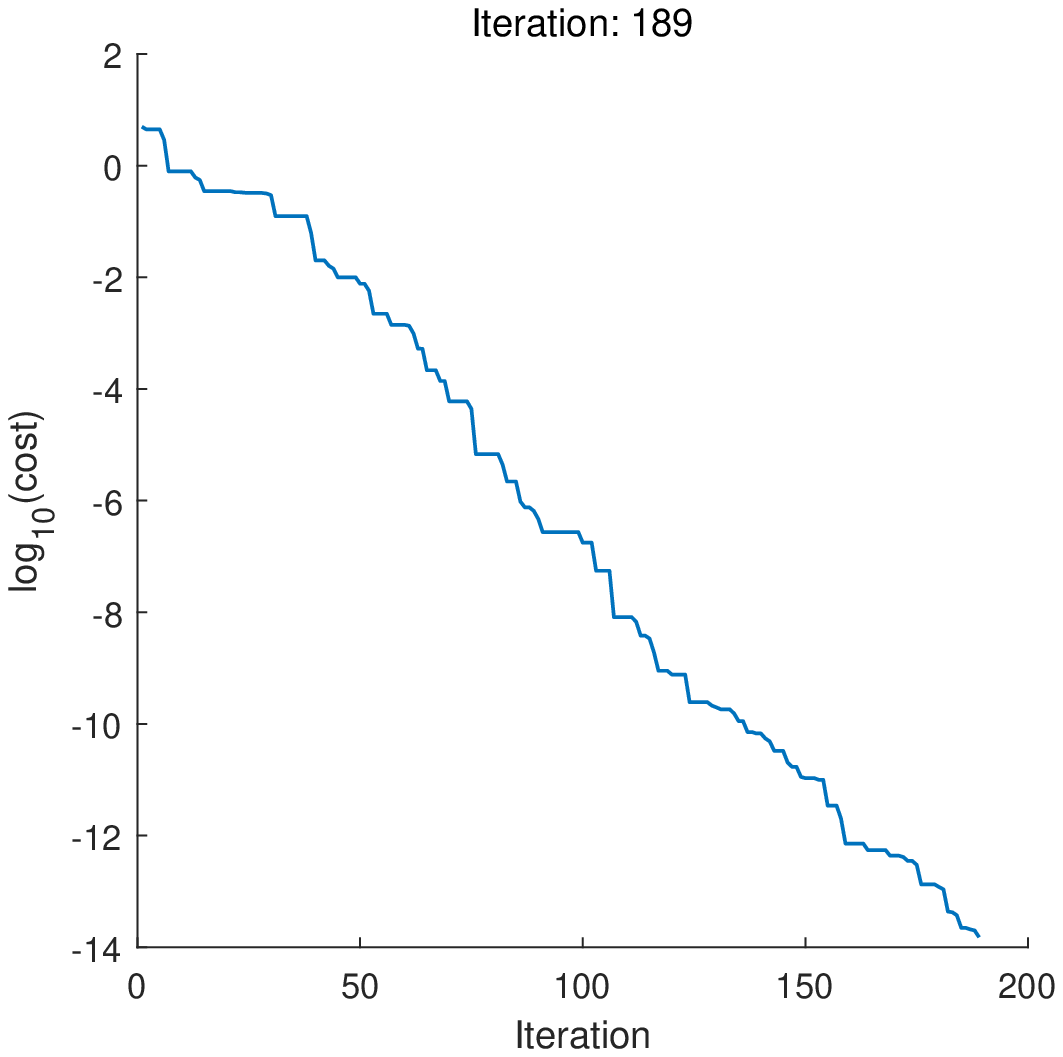}
    \includegraphics[width=0.45\textwidth]{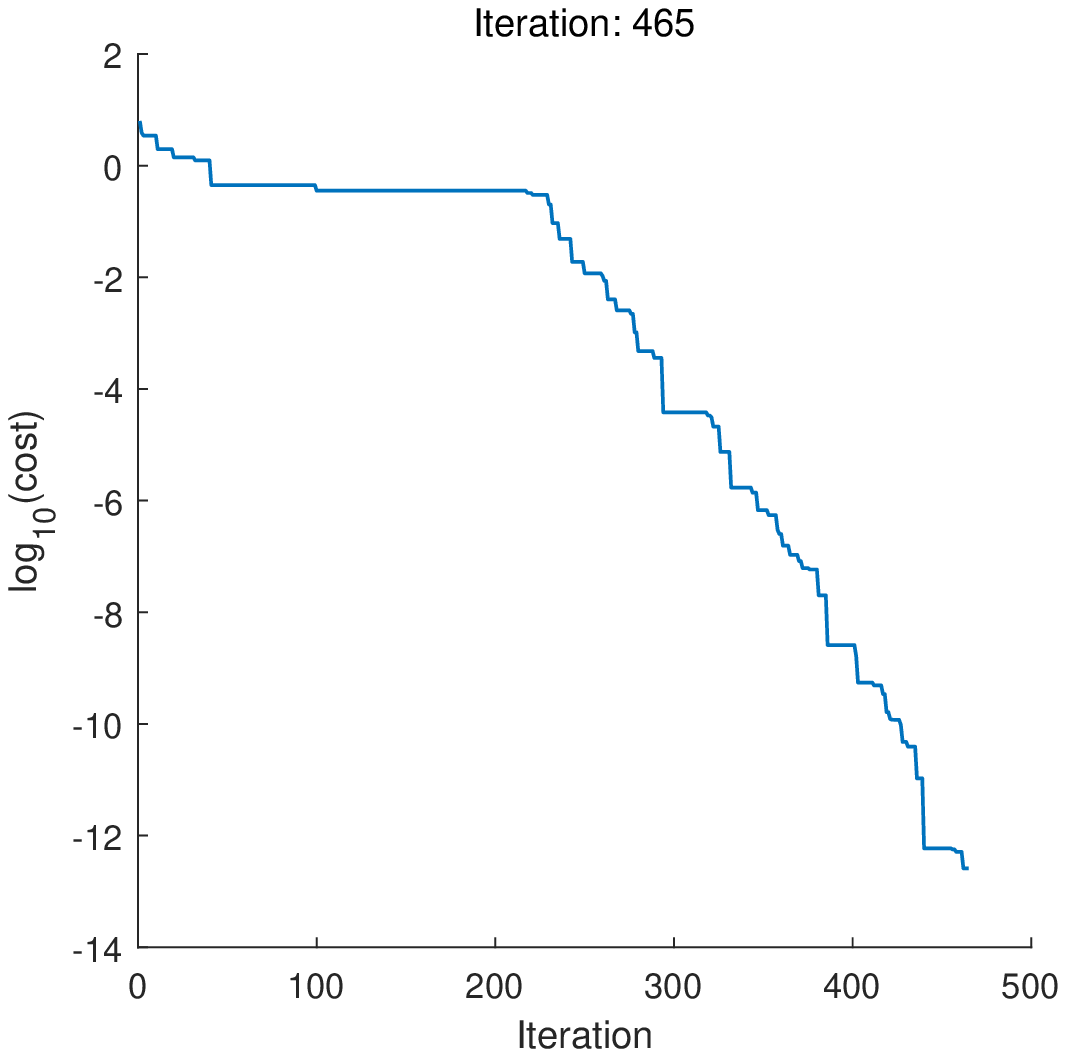}
    \caption{Time-evolution of the best objective value $\min_{j} L({\bold x}_n^j)$ in CBO algorithms with full batch (left) and random batches of $P=10$ (right) for $d=4$ and $N=100$.}
    \label{fig:1}
\end{figure}

Figure \ref{fig:1} shows the dynamics of the current guess of the minimum values in CBO algorithms with full batch (same as $P=100=N$) and random batches ($P=10$) for $d=4$. Since we choose the argument minimum for the weighted average $\bar{\bold x}_n^{S,*}$, the best objective value $\min_{j} L({\bold x}_n^j)$ is not increasing along the optimization step. This property cannot be observed when we use a weighted average as in \eqref{NB-1}. Of course, for large $\beta$, \eqref{NB-1} is close to the argument minimum \eqref{F-1}. We can also check that the speed of consensus is much slower if $P=10$ since the best objective value does not change for a long time (for example, from $n=100$ to $n=200$).

\begin{table}[h!]
\centering
\begin{tabular}{ |c||c|c|c|  }
 \hline
 Success rate & Full batch ($P=100$) & $P=50$ & $P=10$\\
 \hline
 d =  2 & 1.000 & 1.000 & 1.000 \\
 d =  3 & 0.988 & 0.983 & 0.998 \\
 d =  4 & 0.798 & 0.920 & 0.988 \\
 d =  5 & 0.712 & 0.658 & 0.931 \\
 d =  6 & 0.513 & 0.655 & 0.880 \\
 d =  7 & 0.388 & 0.464 & 0.854 \\
 d =  8 & 0.264 & 0.389 & 0.832 \\
 d =  9 & 0.170 & 0.323 & 0.868 \\
 d = 10 & 0.117 & 0.274 & 0.886 \\
 \hline
\end{tabular}
\vspace{0.5em}
\caption{Success rates of CBO algorithms with different dimensions of Rastrigin functions. $1000$ simulations are done for each algorithm.}
\label{table:1}
\end{table}

Table \ref{table:1} presents the success rates of finding the minima $\mathbf{B}$ for different dimensions and algorithms. It clearly shows that the rates get better for smaller $p$, due to the increased randomness. However, the computation with small $p$ takes more steps to converge, as shown in  Figure \ref{fig:2}.

\begin{figure}[h]
    \centering
    \includegraphics[width=0.8\textwidth]{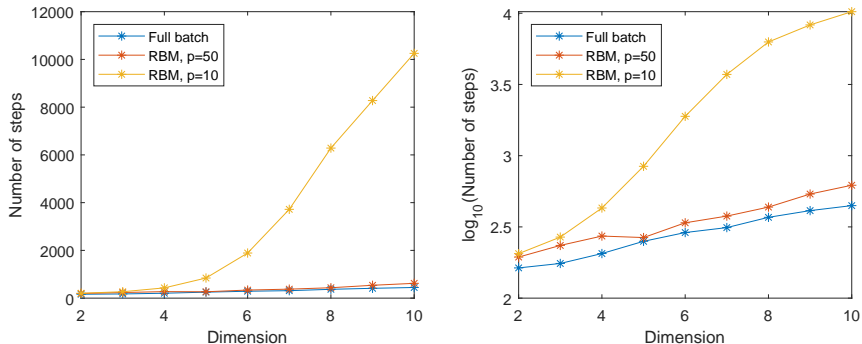}
    \caption{Average number of time steps from full batch and random batches. For each time step, computational time is similar for both methods to converge.}
    \label{fig:2}
\end{figure}
\begin{figure}[h]
    \centering
    \includegraphics[width=0.8\textwidth]{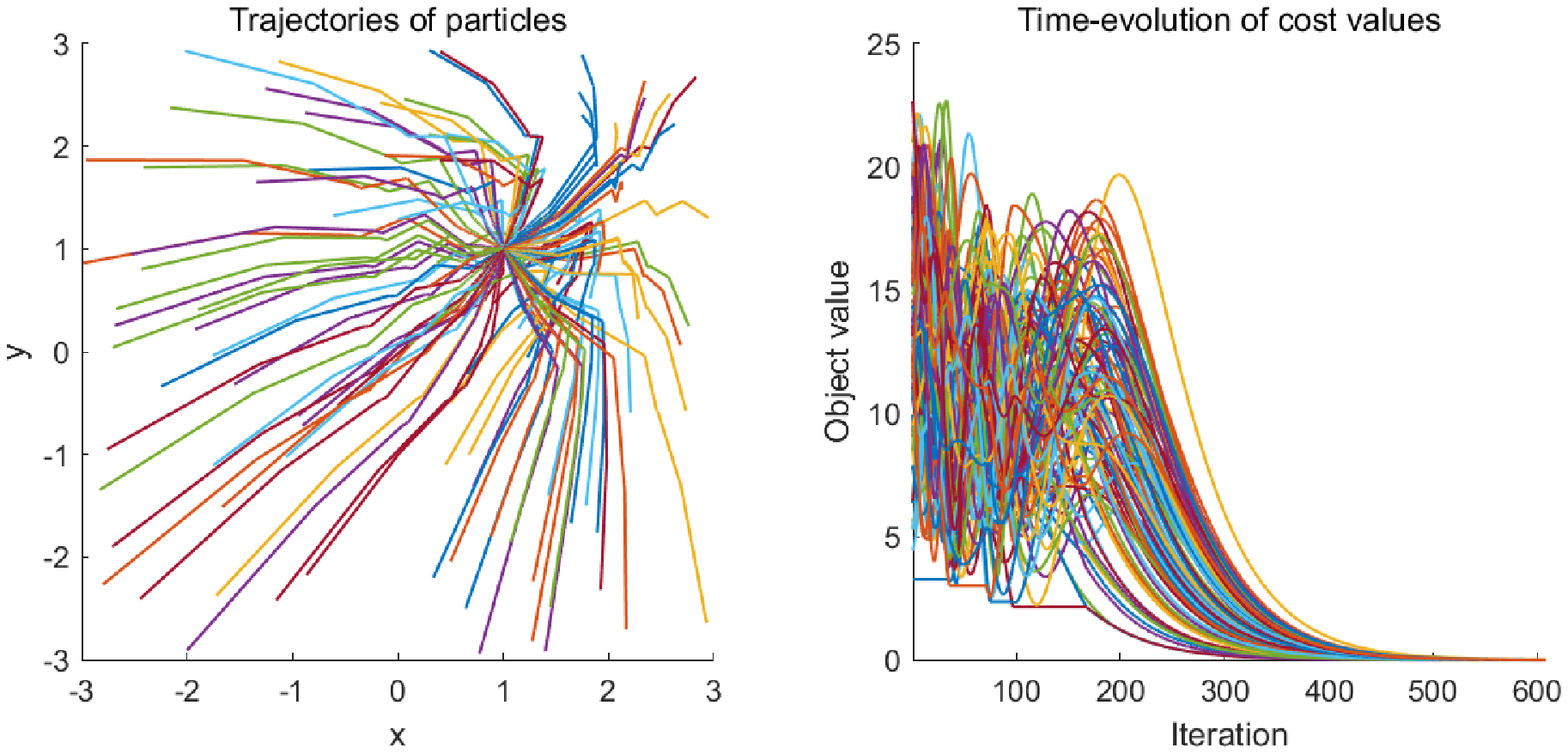}
    \includegraphics[width=0.8\textwidth]{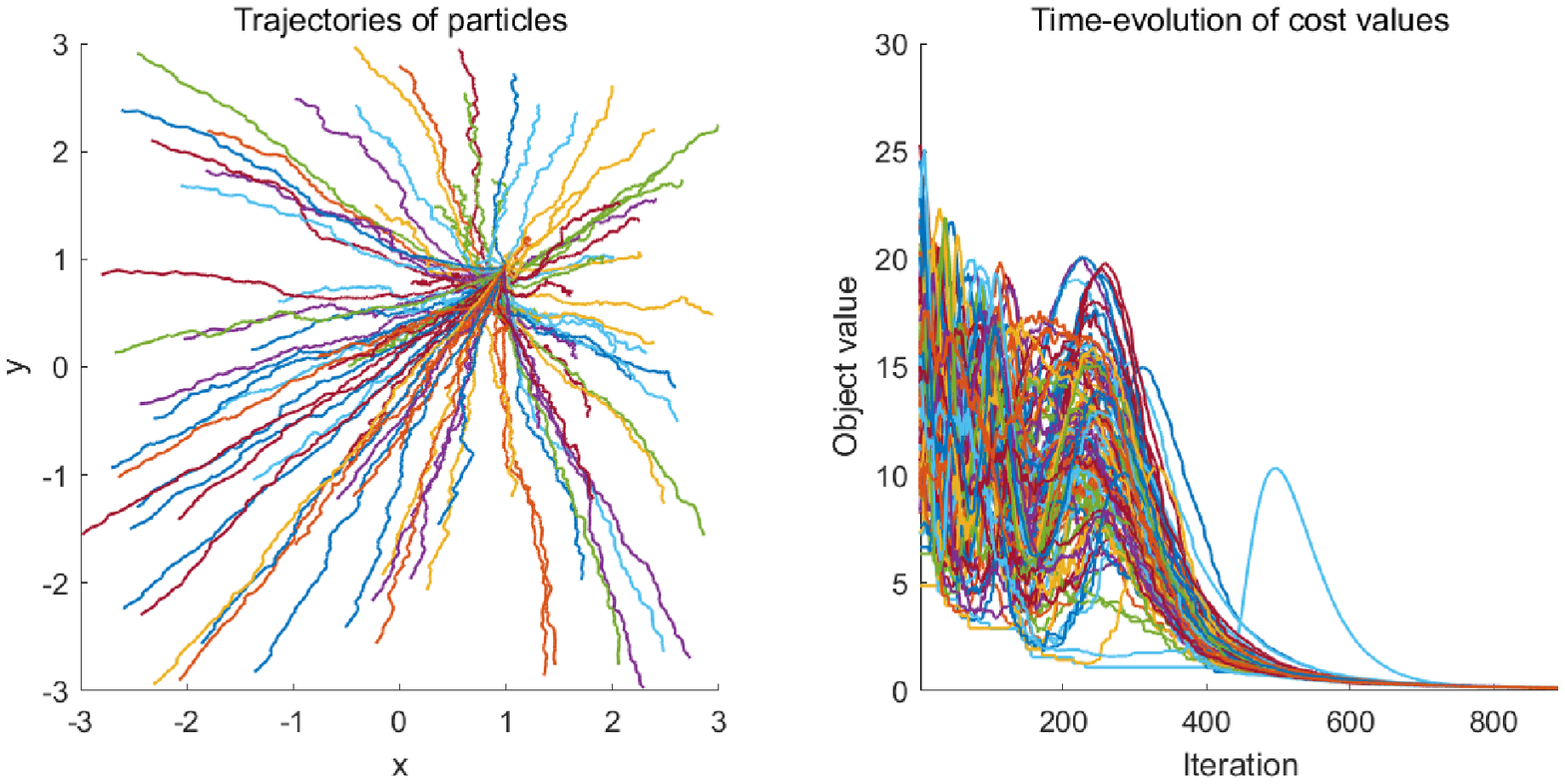}
    \caption{Trajectories of CBO particles (left) and time-evolution of the objective values (right) with $d=4$ without noise from full batch (above) and random batches of $P=10$ (below) represented in the first two dimensions. The trajectories are drawn with the first two components and the objective values are evaluated from $100$ particles. We choose a successful case of optimization for each method to compare convergent behaviors.}
    \label{fig:3}
\end{figure}

Figure \ref{fig:2} shows a weak point of the random batch interactions by showing the number of time steps needed to reach the stopping criterion. Note that the computational costs for each time step is similar for different $P$. However, as it slows down the convergence speed, 
 the number of time steps increases as $N/P$ increases.

Figure \ref{fig:3} shows sample trajectories of the CBO particles with full batch and random batches ($P=10$) for $d=4$ when there is no noise ($\zeta = 0$). If there is noise, then the trajectories nearly cover the area near $0$ so that we cannot catch the differences easily. Note that the random interactions make the particles move around the space. Contrary to the random batch interactions, the particles iwithfull batch dynamics move toward the same point, which is the current minimum position among the particles. This may explain that random batch interactions have better searching ability though it requires more computational cost.

\section{Conclusion} \label{sec:7}
\setcounter{equation}{0}
In this paper, we have provided stochastic consensus and convergence analysis for the discrete CBO algorithm with random batch interactions and heterogeneous external noises. Several versions of discrete CBO algorithm has been proposed to deal with large scale global non-convex minimization problem arising from, for example, machine learning prolems. Recently, thanks to it simplicity and derivative-free nature, consensus based optimization has received lots of attention from applied mathematics community, yet the full understanding of the dynamic features of the CBO algorithm is still far from complete. Previously, the authors investigated consensus and convergence of the modified CBO algorithm proposed in \cite{C-J-L-Z} in the presence of homogeneous external noises and full batch interactions and proposed a sufficient framework leading to the consensus and convergence in terms of system parameters. In this work, we extend our previous work \cite{H-J-Kd} with two new components (random batch interactions and heterogeneous external noises). To deal with random batch interactions, we recast the CBO algorithm with random batch interactions into the consensus model with randomly switching network topology so that we can apply well-developed machineries \cite{D-H-J-K1, D-H-J-K2} for the continuous and discrete Cucker-Smale flockng models using ergodicity coefficient measuring connectivity of network topology (or mixing of graphs). Our analysis shows that stochastic consensus and convergence emerge exponentially fast as long as the variance of external noises is sufficiently small.  Our consensus estimate yields the monotonicity of an objective function along the discrete CBO flow. One should note that  our preliminary result does not yield error estimate for the optimization problem. 
Thus, whether our asymptotic limit point of the discrete CBO flow stays close to the optimal point or not will be an interesting problem  for a future work.

\end{document}